\documentclass[12pt]{amsart}

\usepackage{amsfonts}
\usepackage{amsthm}
\usepackage{amsmath}
\usepackage{amsfonts}
\usepackage{latexsym}
\usepackage{amssymb}
\usepackage[latin1]{inputenc}
\usepackage{verbatim}
\usepackage{multirow}
\usepackage[hypertex]{hyperref}

\numberwithin{equation}{section}

\newcommand{\U}{\mathcal{U}}

\newcommand{\um}{\U_{\cM}}

\def\um{\mathcal{U}_\m}

\newcommand{\ve}{\varepsilon}

\newcommand{\CC}{\mathbb{C}}
\newcommand{\RR}{\mathbb{R}}
\newcommand{\NN}{\mathbb{N}}

\newtheorem{fed}{Definition}[section]
\newtheorem{teo}[fed]{Theorem}
\newtheorem*{teo*}{Theorem}
\newtheorem{lem}[fed]{Lemma}
\newtheorem{cor}[fed]{Corollary}
\newtheorem{pro}[fed]{Proposition}
\theoremstyle{definition}
\newtheorem{rem}[fed]{Remark}

\def\bh{B(\mathcal{H})}

\def\co{\mathrm{conv}\,}

\def\m{\mathcal{M}}
\def\msa{\m^{{\rm sa}}}

\def\cA{\mathcal{A}}
\def\cB{\mathcal{B}}
\def\cC{\mathcal{C}}

\def\cF{\mathcal{F}}

\def\cK{\mathcal{K}}

\def\cM{\mathcal{M}}
\def\cN{\mathcal{N}}

\def\cP{\mathcal{P}}

\def\cT{\mathcal{T}}
\def\cU{\mathcal{U}}

\def\cZ{\mathcal{Z}}

\DeclareMathOperator{\Tr}{{\rm Tr}\,}

\newcommand{\supp}{\mbox{supp }}
\newcommand{\ds}{\displaystyle}

\newcommand{\sess}{\sigma_{\rm e}}
\newcommand{\sdisc}{\sigma_{\rm d}}

\newcommand\e{\varepsilon}
\newcommand{\diag}[1]{\hbox{diag}\left( #1\right)}
\newcommand{\falpha}{f}
\newcommand{\gbeta}{g}

\begin{document}

\title{Schur-Horn theorems in II$_\infty$-factors}

\author{Martín Argerami}
\author{Pedro Massey}

\subjclass[2010]{Primary 46L51; Secondary 46L10, 52A05, 15A18}

\begin{abstract} We describe majorization between selfadjoint operators in
a $\sigma$-finite II$_\infty$ factor $(\m,\tau)$ in terms of simple
spectral relations.
For a diffuse abelian von Neumann subalgebra $\cA\subset \m$ with trace-preserving
conditional expectation $E_\cA$,
we characterize the closure in the measure topology of the image through $E_\cA$
of the unitary orbit of a selfadjoint operator in $\m$ in terms of majorization (i.e., a Schur-Horn
theorem).
We also obtain similar results for the contractive orbit of positive
operators in $\m$ and for the unitary
and contractive
orbits of $\tau$-integrable operators in $\m$.
\end{abstract}

\maketitle

\section{Introduction}

Given two vectors $x,y\in\RR^n$, we say that $x$ is \emph{majorized} by $y$ ($x\prec y$)
if
\[
\sum_{j=1}^kx_j^\downarrow\leq\sum_{j=1}^ky_j^\downarrow,\ \ k=1,\ldots,n-1;\ \ \
\sum_{j=1}^nx_j=\sum_{j=1}^ny_j\,,
\]where $x^\downarrow\in \RR^n$ denotes the vector obtained from $x$  by re-arranging the entries in non-increasing order.
The first systematic study of the notion of majorization is attributed to Hardy, Littlewood, and Polya
\cite{HardyLittlewoodPolya1929}. We refer the reader to \cite{Bhatia1997} for further
references and properties of majorization. It is well known that (vector) majorization is
intimately related with the theory of doubly stochastic matrices. Indeed,
$x\prec y$ if and only if $x=Dy$ for some doubly stochastic matrix $D$;
then, as a consequence of Birkhoff's characterization of the extreme points of the set of
doubly stochastic matrices \cite{Birkhoff1946}, one can conclude that
\begin{equation}\label{eq:birkhoff en dimension finita}
\{x\in\RR^n:\ x\prec y\}=\co\{y_\sigma:\ \sigma \in\mathbb{S}_n\}\,,
\end{equation} where $\co\{y_\sigma:\ \sigma \in\mathbb{S}_n\}$ denotes the convex hull of the set of
vectors $y_\sigma$ that are obtained from $y$ by re-arrangement of its components
through permutations $\sigma\in \mathbb{S}_n$.

It turns out that majorization also characterizes the relation between the spectrum
and the diagonal of a selfadjoint matrix. Let $M_n(\CC)$ denote the algebra of complex
$n\times n$ matrices. For $A\in M_n(\CC)$, let $\diag{A}=(a_{11},a_{22},\ldots,a_{nn})\in\CC^n$,
and let $\lambda(A)\in\CC^n$ be the vector whose coordinates are the
eigenvalues of $A$, counted with multiplicity. I. Schur \cite{Schur1923}
 proved that for $A\in M_n(\CC)$ selfadjoint,
$\diag{A}\prec\lambda(A)$; while A. Horn  \cite{Horn1954} proved the
converse: given $x,y\in\RR^n$ with $x\prec y$, there exists a selfadjoint matrix $A\in
M_n(\CC)$, with $\diag{A}=x$, $\lambda(A)=y$. For $y\in \CC^n$ let $M_y\in M_n(\CC)$ denote the diagonal matrix with main diagonal $y$ and let $\cU_n\subset M_n(\CC)$ denote the group of unitary matrices.
The results from Schur and Horn can then be combined in the following assertion: given $y\in \RR^n$,
\begin{equation}\label{eq:sh en dimension finita}
\{x\in\RR^n:\ x\prec y\}=\{\diag{U\,M_{y}\,U^*}:\ U\in\cU_n\},
\end{equation}
usually known as the Schur-Horn Theorem. The fact that majorization relations imply
a family of entropic-like inequalities makes the Schur-Horn theorem an important
tool in matrix analysis theory \cite{Bhatia1997}. It has also been observed that the
Schur-Horn theorem plays a crucial role in frame theory
\cite{AntezanaMasseyRuizStojanoff2009,MR2176806,MR2581231}.

Majorization in the context of von Neumann algebras has been widely
studied (see for instance \cite{ArgeramiMassey2008b,Hiai1987,Hiai1989,
HiaiNakamura1987,Kamei1983,Kamei1984}).
In \cite{Hiai1987} and \cite{Hiai1989} F. Hiai showed several characterizations of
majorization
in a semifinite von Neumann algebra, including a generalization of
\eqref{eq:birkhoff en dimension finita}, i.e. a ``Birkhoff'' theorem. Nevertheless, \
the lack of the corresponding ``Schur-Horn'' theorems in the
general context of von Neumann factors was only recently observed.
Early work on this topic was developed by A. Neumann \cite{Neumann1999,MR1887630}
in relation with an extension to infinite dimensions of the linear Kostant convexity
theorem in Lie theory.

It was in \cite{ArvesonKadison2007} that W. Arveson and R.V. Kadison
conjectured a Schur-Horn theorem in II$_1$ factors. Although this conjecture remains
an open problem, there has been progress on related (but weaker) Schur-Horn theorems
in this context \cite{ArgeramiMassey2007,ArgeramiMassey2008a,ArgeramiMassey2009}.
There has also been significant improvements of Neumann's work on majorization between sequences in $c_0(\RR^+)$ due to V. Kaftal and G. Weiss \cite{MR2436756,Kaf2} because of the relations between infinite dimensional versions of the Schur-Horn theorem (via majorization of bounded structured real sequences) and arithmetic mean ideals (see also \cite{ArvesonKadison2007} for improvements in the compact case in $B(H)$).

In this paper we prove versions of the Schur-Horn theorem (i.e.
generalizations of \eqref{eq:sh en dimension finita}) in the case
of a $\sigma$-finite II$_\infty$-factor. These results extend those
obtained in
\cite{ArgeramiMassey2007,ArgeramiMassey2008a,Neumann1999}. Our
results are in the vein of Neumann's work, and they are related
with a weak version of Arveson-Kadison's scheme for Schur-Horn theorems, but
modeled in II$_\infty$ factors.  These extensions are formally
analogous to the Schur-Horn theorems in
\cite{ArgeramiMassey2007,ArgeramiMassey2008a}, but the techniques are
more involved in the infinite case. We show that our results
are optimal, in the sense that they can not be strengthened for a
general selfadjoint operator in a II$_\infty$ factor.

The paper is organized as follows.  In section \ref{section:preliminaries} we develop
notation and some basic results on the measure topology and the $\tau$-singular values in von Neumann algebras. Section \ref{section: B(H)} deals with majorization in $B(H)$, including some results complementing those in
\cite{Neumann1999}.
In Section \ref{section:majorization} we consider a notion of majorization between
selfadjoint operators in a II$_\infty$ factor $(\m,\tau)$
-- in line with Neumann's idea \cite{Neumann1999} -- together with several of
its basic properties. Although  majorization in II$_\infty$ factors is not a new notion
\cite{Hiai1987,Hiai1989}, our approach is quite different from the previous
presentations. In section \ref{section:Schur-Horn Theorem} we
state and prove the generalizations of the Schur-Horn theorem in II$_\infty$ factors.
Our strategy is to reduce the problem to a discrete version, where we can apply the
Schur-Horn theorems developed in Section \ref{section: B(H)} for $B(H)$. We then proceed to show that
Hiai's notion of majorization  in terms of Choquet's theory of comparison of
measures \cite{Hiai1989} coincides with ours. We finally consider similar results
for the contractive orbit of a positive operator  and for the unitary and contractive orbits of
bounded $\tau$-measurable operators.

\section{Preliminaries}\label{section:preliminaries}

  Let $(\m,\tau)$ be a $\sigma$-finite, semi-finite, diffuse von Neumann algebra.
  The real subspace of selfadjoint elements in $\m$ is denoted by $\m^{\rm sa}$;
  the group of unitary operators by $\um$; and the set of selfadjoint projections by $\cP(\m)$.
Given $p\in\cP(\m)$, we use the notation $p^\perp=I-p$.
  For any $a\in\m^{\rm sa}$ and
any Borel set $\Delta\subset\RR$, $p^a(\Delta)\in \cP(\m)$ denotes the spectral projection
of $a$ corresponding to
$\Delta$.

In \cite{Fack1982} T. Fack considered in $\m$ the ideals $\cF(\m)=\{x\in
\m:\ \tau(\supp x^*)<\infty\}$ -- the  \emph{$\tau$-finite
rank operators} -- and $\cK(\m)=\overline{\cF(\m)}$, the ideal of \emph{$\tau$-compact operators}.
The quotient
C$^*$-algebra $\m/\cK(\m)$ is called the generalized Calkin
algebra. The \emph{essential spectrum} of $x$ -- denoted
$\sess(x)$ -- is the spectrum of $x+\cK(\m)$ as an element of
$\m/\cK(\m)$. The complement of $\sess(x)$ within $\sigma(x)$ is
 the \emph{discrete spectrum} $\sdisc(x)$ of $x$. As shown in
\cite{Hiai1989},  for $x\in \m^{\rm sa}$,
\[
\sess(x)=\{t\in\sigma(x):\ \forall\e>0,\ \tau(p^x(t-\e,t+\e))
=\infty\}.
\]
It follows from the previous definitions that
$x\in  \m^{\rm sa}$ is $\tau$-compact if and only if $\sess(x)=\{0\}$.

We  consider in $\m$ the \emph{measure topology} $\cT$, which
is the linear topology given by the neighborhoods of $0\in \m$,
\[
V(\e,\delta)= \{ r\in \m: \ \exists p\in \cP(\m),\
\|rp\|< \e, \ \tau(p^\perp)<\delta\},
\]
where $\e,\delta>0$. For a II$_1$ factor, $\cT$ reduces to the $\sigma$-strong topology
on bounded sets,
while in a type I$_\infty$ factor it reduces to the norm topology.

\begin{fed}\label{def:upper and lower spectral scale}
The \emph{upper spectral scale} of\ \ $b\in\m^{\rm sa}$ is the
non-increasing right-continuous real function
\[
\lambda_t(b)=\min\{s\in\RR:\ \tau(p^b(s,\infty))\leq t\},\ \ \ t\in[0,\infty).
\] The \emph{lower spectral scale} of $b$ is the non-decreasing right-continuous function
\[
\mu_t(b)= - \lambda_t(- b)= \max\{s\in\RR:\ \tau(p^b(-\infty,s))\leq t\},\ \ \ t\in[0,\infty).
\]
\end{fed}

A direct consequence of these definitions is that  $\lambda_t(b),\,\mu_t(b)\in \sigma(b)$
for every $t\in \RR^+$.
The function $t\mapsto\lambda_{t}(b)$ is the analogue of the re-arrangement of
the eigenvalues (in non-increasing order and counting multiplicities) of a self-adjoint matrix.

For $x\in \m$ we can consider the \emph{$\tau$-singular values} of $x$ given by
$\nu_t(x)=\lambda_t(|b|)$, $t\in[0,\infty)$. The spectral scale and $\tau$-singular
values have been extensively studied
\cite{Fack1982,FackKosaki1986,HiaiNakamura1987,Kadison2004,Petz1985} in the broader context
of $\tau$-measurable operators affiliated to $(\m,\tau)$.

The elements of $\cK(\m)$ can be described in terms of $\tau$-singular values.
Indeed,  $x\in \m$ is $\tau$-compact if and only if
$\lim_{t\rightarrow\infty}\nu_t(x)=0$ \cite{Hiai1987}.
We will make frequent use of the fact that (since $\m$ is diffuse) a given $\tau$-compact
$x\in\m^{+}$  admits a {complete flag}, i.e. an increasing assignment
$\RR^+\ni t\mapsto e(t)\in\cP(\m)$ such that $\tau(e(t))=t$, and
\begin{equation}\label{eq:integral con la bandera completa}
x=\int_0^\infty\,\lambda_t(x)\,de(t)\,.
\end{equation}
As opposed to the finite case \cite{ArgeramiMassey2007}, the equality in \eqref{eq:integral con la bandera completa} does not hold for arbitrary $\tau$-compact selfadjoint operators in $\m$. This is possibly one of the reasons why majorization has  been considered mainly between positive operators in the semi-finite algebras
(see the remarks at the end of \cite{Hiai1987}).
We shall overcome this issue by considering both the upper and lower spectral scale, as done
in \cite{Neumann1999} in the case of separable I$_\infty$ factors.

The following fact is used in \cite{Hiai1989} (in the context of possibly unbounded operators)
but we do not know of an explicit proof in the literature. For $x\in \m$, we denote its usual one-norm or trace norm in $(\m,\tau)$ by $\|x\|_1=\tau(|x|)\in [0,\infty]$.

\begin{pro}\label{proposition:la top de la medida esta dada por una norma en acotados}
Let $(\m,\tau)$ be a semifinite von Neumann algebra. For $s>0$ let $\|\cdot\|_{(s)}$ be the norm given by
\[
\|x\|_{(s)}=\inf\{\|x_1\|_1^{\phantom1}+s\,\|x_2\|:\ x=x_1+x_2,\ x_1,x_2\in\m\},\ \ x\in\m.
\]
Then $\|x\|_{(s)}=\int_0^s\nu_t(x)\,dt$, and the topology induced by $\|\cdot\|_{(s)}$ agrees
with the measure topology on bounded sets.
\end{pro}
\begin{proof}
The equality $\|x\|_{(s)}=\int_0^s\nu_t(x)\,dt$ is proven in \cite{FackKosaki1986}
in the argument after Theorem
4.4. We now show that the topology induced by $\|\cdot\|_{(s)}$
and the measure topology agree on bounded sets. Indeed, if $0<s\leq r$ then there exists $k\in \NN$ such that $r\leq k\,s$ and therefore $\|x\|_{(s)}\leq \|x\|_{(r)}\leq k\, \|x\|_{(s)}$, since $t\mapsto\nu_t(x)$ is a non-increasing function. This shows that the norms $\|\cdot\|_{(s)}$, for $s>0$, are all equivalent and induce the same topology.
Hence we can assume without loss of generality that $s=1$.

 If $\|x\|_{(1)}<d$, then $\int_0^1\nu_t(x)\,dt<d$. Using that $\nu_t(x)$ is non-increasing, there
exists $t_0$ with $0<t_0<\sqrt{d}$ such that $\nu_{t_0}(x)<\sqrt{d}$.
By \cite[Proposition 2.2]{FackKosaki1986},
\begin{equation}\label{eq:definicion del lambda con la norma para positivos}
\nu_{t_0}(x)=\inf\{\|xq\|:\ \tau(q^\perp)\leq t_0\},
\end{equation}
 so there is a projection $q\in\cP(\m)$ such that $\|xq\|<\sqrt{d}$ and $\tau(q^\perp)
<\sqrt{d}$; that is, $x\in V(\sqrt{d},\sqrt{d})$. Thus every ball in the $\|\cdot\|_{(1)}$-topology
lies inside a neighborhood of 0 in the measure topology.

Conversely, if $x\in V(\e,\delta)$ and $\|x\|\leq k$,  there exists a projection $q\in\cP(\m)$
such that $\|xq\|<\e$, $\tau(q^\perp)<\delta$. Since $x=xq^\perp+xq$,
\[
\|x\|_{(1)}\leq\|xq^\perp\|_1^{\phantom1}+\|xq\|\leq k\delta+\e;
\]
that is, $V(\e,\delta)\cap \{x\in\m:\ \|x\|\leq k\}\subset \{x\in\m:\ \|x\|_{(1)}\leq k\delta+\e\}$.
\end{proof}

\begin{cor}\label{lem:topologia de la medida en un II_1}
Let $\cN$ be a II$_1$-factor with trace $\tau_\cN$, and let $\{x_j\}$ be a
bounded net. Then $x_j\xrightarrow{\|\cdot\|_1}x$ if and only if
$x_j\xrightarrow\cT x$.
\end{cor}
\begin{proof}
For any $x\in\cN^{\rm sa}$ we have
$\|x\|_1=\tau_\cN(|x|)=\int_0^1\nu_t(x)\ ds$. Then $\|\cdot\|_1=\|\cdot\|_{(1)}$ and
Proposition
\ref{proposition:la top de la medida esta dada por una norma en acotados} yields
the result.
\end{proof}

We will often and without mention make use of the following properties of the measure topology.

\begin{cor}\label{lem:propiedades de T}Let $\cA\subset \m$  be a von Neumann subalgebra and
let $E_\cA$ be the trace preserving conditional expectation onto $\cA$.
Let $\{x_j\}\subset\m^{\rm sa}$, $\alpha,\beta\in\RR$ with $\alpha\,I\leq x_j\leq\beta\,I$ for
every $j$ and such that $x_j\xrightarrow\cT x$. Then
\begin{enumerate}
\item\label{lem:propiedades de T:1} $x\in\m^{\rm sa}$, and $\alpha\leq x\leq\beta$.
\item\label{lem:propiedades de T:3} $E_\cA(x_j)\xrightarrow\cT E_\cA(x)$.
\end{enumerate}
\end{cor}
\begin{proof}
In order to prove \eqref{lem:propiedades de T:1} first notice that if
$x_j\xrightarrow\cT x$ with $x_j\geq0$ for every $j$ then $x\in \m^{\rm sa}$; indeed, this follows from the facts that the operation of taking adjoint is continuous in the measure topology and that this topology is Hausdorff.
If $x\not\in\cM^+$, there exists a nonzero
projection $q\in\cM$
and $k\in\RR^+$ such that $q\,x\,q\leq(-k)\,q$. By replacing $q$ by a
smaller projection if necessary, we may assume
that $\tau(q)<\infty$. We have $q\,x_j\,q\xrightarrow\cT q\,x\,q$, so for $j$
 big enough
there exists
a projection $p$ such that $\|(q\,x\,q-q\,x_j\,q)p\|<k/3$ and $\tau(p^\perp)<\tau(q)/2$. Then $pqp\ne0$, since
\[
\tau(p\,q\,p)=\tau(p\,q)=\tau(q)-\tau(p^\perp q)\geq\tau(q)-\tau(q)/2=\tau(q)/2>0.
\]
We also get from above that $\tau(q)\leq2\tau(p\,q\,p)$.
But then $\tau(p\,q\,(x_j-x)\,q\,p)=\tau(q\,[q\,(x_j-x)\,q\,p])\leq\frac{k}3\tau(q)$, so
\begin{align*}
0&\leq\tau(p\,q\,x_j\,q\,p)=\tau(p\,q\,x\,q\,p)+\tau(p\,q\,(x_j-x)\,q\,p)\leq(-k)\tau(p\,q\,p)+\frac{k}3\,\tau(q)\\
&\leq(-k)\tau(p\,q\,p)+\frac{2k}3\tau(p\,q\,p)=-\frac{k}3\,\tau(p\,q\,p)<0,
\end{align*}
a contradiction. This shows that $x\ge0$. By linearity we get that
if $x_j\xrightarrow\cT x$ and
$\alpha\leq x_j\leq\beta$ then $\alpha\leq x\leq\beta$.

Item \eqref{lem:propiedades de T:3} follows from the fact that $E_\cA$ is contractive
with respect to $\|\cdot\|_{(1)}$ together with Proposition \ref{proposition:la top de la medida esta dada por una norma en acotados}. Indeed, it  is well known that $\|E_\cA(x)\|\leq\|x\|$ for $x\in \m$. Using that
$\tau(E_\cA(x)\,y)=\tau(x\,E_\cA(y))\leq\|E_\cA(y)\|\,\tau(|x|)$
we get
\[
\|E_\cA(x)\|_1=\sup\{|\tau(E_\cA(x)\,y)|:\ y\in\m,\ \|y\|\leq1\}\leq\|x\|_1.
\] For any decomposition $x=y+z$, since $E_\cA(x)=E_\cA(y)+E_\cA(z)$,
\[
\|E_\cA(x)\|_{(1)}\leq\|E_\cA(y)\|_1+\|E_\cA(z)\|\leq\|y\|_1+\|z\|.
\]So, by Proposition
\ref{proposition:la top de la medida esta dada por una norma en acotados},
$\|E_\cA(x)\|_{(1)}\leq\|x\|_{(1)}$ for all $x\in\m$, and so $E_\cA$ is $\cT$-continuous.
\end{proof}

\section{Majorization in $\ell^\infty(\NN)$ and $B(H)$ revisited}\label{section: B(H)}

Let $H$ be a complex separable Hilbert space. In this section we revise and complement A. Neumann's
\cite{Neumann1999} theory on majorization
between self-adjoint operators in $B(H)$.
These results will play a key role in our proof of the Schur-Horn theorem in
II$_\infty$-factors (Theorem \ref{teo:SH}). For conceptual and notational convenience,
we shall follow the exposition in \cite{AntezanaMasseyRuizStojanoff2009} (see also \cite{Kadison2004}).

In $B(H)$ we consider the canonical trace $\Tr$.
We write $\U(H)$  for the group of unitary operators in $H$,
and $\cC(H)$ for the semigroup of contractive operators in $B(H)$,
i.e.
\[\cC(H)=\{v\in B(H):\ v^*v\leq I \}.
\]

For $k\in \NN$, let $\cP_k$ be the set of orthogonal projections $p\in B(H)$ such that $\Tr(p)= k$.
For $b\in B(H)^{\rm sa}$, $k\in \NN$, we consider
\begin{equation}\label{eq:definicion de U_k y L_k en B(H)}
U_k(b)=\sup_{p\in \cP_k}\Tr(b\,p),\ \text{ and } \
L_k(b)=\inf_{p\in \cP_k}\Tr(b\, p).
\end{equation}
For each $k\in \NN$, both $b\mapsto U_k(b)$ and $b\mapsto L_k(b)$ are norm-continuous
in $B(H)$, with $L_k(b)=-U_k(-b)$. Moreover,
$U_k(u^*\, b\,u)=U_k(b)$ for every $b\in B(H)^{\rm sa}$,
$u\in \U(H)$.

Following \cite{Neumann1999} (but with a different notation) we define, for
$f\in\ell^\infty(\NN)$ and $k\in\NN$,
\begin{equation}\label{eq:definicion de U_k y L_k en ell infty}
U_k(f)=\sup\{\sum_{j\in K}f_j:\ |K|=k\}, \ \ \
L_k(f)=\inf\{\sum_{j\in K}f_j:\ |K|=k\}.
\end{equation}
Again, for each  $k\in \NN$, $L_k(f)=-U_k(-f)$. The similarity of the notations in
\eqref{eq:definicion de U_k y L_k en B(H)} and
\eqref{eq:definicion de U_k y L_k en ell infty} is justified by the following fact: if
$b\in\bh$ is selfadjoint and such that there exists an orthonormal basis
$\{e_i\}_{i\in \NN}$ of $H$ and $f=(f_i)_{i\in \NN}\in \ell^\infty_\RR(\NN)$ such that
 $be_i= f _i\,e_i$, $i\in \NN$ (i.e. if $b$ is diagonal), then
by \cite[Proposition 3.3]{AntezanaMasseyRuizStojanoff2009}
 \begin{equation} \label{eq:caracterizacion del U_k de Neumann}
 U_k(b)=U_k(f),\ \
L_k(b)=L_k(f), \ \ \ k\in \NN\,.\end{equation}
\begin{fed}[Operator majorization in $B(H)$ \cite{AntezanaMasseyRuizStojanoff2009}]
\label{def:mayorizacion de Neumann}
Let  $a,\,b\in B(H)^{\rm sa}$. We say that:
\begin{enumerate}
\item $a$ is \emph{submajorized} by $b$,
denoted $a\prec_w b$, if $U_k(a)\leq U_k(b)$ for every $k\in \NN$;
\item $a$ is \emph{majorized} by $b$, denoted $a\prec b$, if $a\prec_w b$ and $L_k(a)\geq L_k(b)$ for every $k\in \NN$.
\end{enumerate}
\end{fed}
We will also use the notion of vector majorization in $\ell^\infty_\RR(\NN)$ (used
implicitly in \cite{Neumann1999}) as follows:
\begin{fed}[Vector majorization in $\ell^\infty_\RR(\NN)$]
\label{def:mayorizacion de Neumann original}
Let $f,\, g\in \ell^\infty_\RR(\NN)$. We say that:
\begin{enumerate}
\item $f$ is \emph{submajorized} by $g$, denoted $f\prec_w g$ if $U_k(f)\leq U_k(g)$ for every $k\in \NN$;
\item $f$ is \emph{majorized} by $g$, denoted $f\prec g$, if $f\prec_w g$ and $L_k(f)\geq L_k(g)$ for
every $k\in\NN$.
\end{enumerate}
\end{fed}

We fix an orthonormal basis $\mathcal B=\{e_i\}_{i\in \NN}$ on $H$, with associated
system of matrix units
$\{e_{ij}\}_{i,j\in \NN}$ in $B(H)$. For each $f\in \ell^\infty(\mathbb N)$ we denote by $M_f\in B(H)$
the induced diagonal operator with respect to $\mathcal B$, i.e.
$M_f=\sum_{i\in \NN} f_i\, e_{ii}$. By \eqref{eq:caracterizacion del U_k de Neumann},
it is immediate that for all $f,g\in\ell^\infty_\RR(\NN)$,
\begin{equation}\label{eq:equivalencia de las dos mayorizaciones}
M_f\prec M_g\iff f\prec g,\ \ \ M_f\prec_w M_g\iff f\prec_w g.
\end{equation}

We denote by $P_D:B(H)\to B(H)$ the trace
preserving conditional expectation onto the (discrete) diagonal masa with
respect to the fixed orthonormal basis. Explicitly, for each $x\in B(H)$,
\[
P_D(x)=\sum_i e_{ii}\,x\,e_{ii}=\sum_{i} f_i\ e_{ii}=M_f\,,\ \  \
\mbox{where }f_i=\langle x\,e_i,e_i\rangle, \ i\in \NN.
\]

We will use the following result from Neumann, which is a combination of
\cite[Theorem 2.18]{Neumann1999} and \cite[Theorem 3.13]{Neumann1999}.
Although Neumann's result is phrased in terms of vectors in $\ell^\infty_\RR(\NN)$,
we phrase the result in terms
of operators in $B(H)$,
as in \cite[Theorem 3.10]{AntezanaMasseyRuizStojanoff2009}.

\begin{teo}[A Schur-Horn theorem for $B(H)$]\label{theorem:Neumann 2.18-3.13}
Let $H$ be a separable complex Hilbert space and
let $P_D$ denote the unique trace preserving conditional expectation onto the discrete
masa of diagonal operators with respect to the orthonormal basis $\cB$ of $H$.
Then, for $b\in B(H)^{\rm sa}$,
\[\overline{\{P_D(u \, b\, u^*): \ u\in \U(H)\}}^{\, \|\, \|}
=\{M_f:\ f\in \ell^\infty_\RR(\NN),\, M_f\prec b \}\, .
\] \end{teo}

\bigskip

As a consequence of Theorem \ref{theorem:Neumann 2.18-3.13} and \eqref{eq:equivalencia de las dos mayorizaciones}
we recover Neumann's result for majorization in $\ell^\infty_\RR(\NN)$
 which states that, for
$f,\,g\in \ell^\infty_\RR(\NN)$,
\begin{equation}\label{theorem:Neumann 2.18-3.13d}
M_f\in \overline{\{P_D(u\,  M_g \, u^*): \ u\in \U(H)\}}^{\, \|\, \|} \
\text{ if and only if } \ f\prec g\,.
\end{equation}

\medskip

In the rest of this section we will
develop a contractive version of Theorem \ref{theorem:Neumann 2.18-3.13} for positive
operators of $B(H)$ (Theorem \ref{theorem:neumann contractivo}). We will need a few
preliminary results.

 A proof of the following elementary inequality can be found in \cite[Lemma 24]{Kadison2004}.
\begin{lem}\label{lemma:la desigualdad numerica} Let $y_1\geq y_2\geq\cdots$
be positive real numbers and $\alpha_1,\alpha_2,\ldots\in[0,1]$ with
$\sum_{j=1}^\infty\alpha_j\leq k$. Then
\begin{equation}\label{eq:la desigualdad numerica}
\sum_{j=1}^\infty\alpha_j\ y_j\leq\sum_{j=1}^ky_j.
\end{equation}
\end{lem}

\begin{lem}\label{lemma:U_k contractivo}
For any $\gbeta\in\ell^\infty(\NN)^{+}$, $k\in\NN$ we have
\[
U_k(\gbeta)=\sup\{\Tr(M_\gbeta\,x):\ x\in \cC(H)^+,\ \Tr(x)\leq k\}\,.
\]
\end{lem}
\begin{proof}
The inequality ``$\leq$'' is clear by \eqref{eq:definicion de U_k y L_k en B(H)}
and \eqref{eq:caracterizacion del U_k de Neumann}.
To prove the reverse inequality, fix $k\in\NN$,
let $\e>0$, and fix $x\in \cC(H)^+$ with $\Tr(x)\leq k$. As $x$ is a compact and positive contraction,
$x=\sum_j\gamma_jh_j$, where $\{h_j\}_j$ is a pairwise-orthogonal
family of rank-one projections, $0\leq\gamma_j\leq 1$ for all $j$, and $\sum_j\gamma_j\leq k$.
We also have that $M_\gbeta=\sum_i\gbeta_ie_{ii}$, where $\{e_{ii}\}_i$ is
the pairwise-orthogonal family of rank-one projections associated with the canonical basis $\cB$.
Let $\beta=\limsup_{n}g_n= \max\sess(M_g)$ and define $\gbeta'\in\ell^\infty(\NN)$ by
\[
\gbeta'_i
=\left\{\begin{array}{ll}\gbeta_i&\mbox{ if }\gbeta_i\geq\beta+\e \\ \beta&\mbox{ otherwise}
\end{array}\right.
\]
Using \cite[Lemma 2.17]{Neumann1999} it is readily seen that $|U_k(\gbeta')-U_k(\gbeta)|<k\e$.
Since the entries in $\gbeta$ that are strictly greater than $\beta$ can only appear a
finite number of times, we have that the set $D=\{i:\ \gbeta'_i>\beta\}$ is finite. So there is a
unitary $u\in \U(H)$ (induced by an appropriate permutation) such that $\gbeta''$ given by
$M_{\gbeta''}=uM_{\gbeta'}u^*$ satisfies $\gbeta''_1\geq \gbeta''_2\geq\cdots\geq
\gbeta''_m$, where $m=|D|$, and $\gbeta''_i=\beta$ if $i>m$.
For each $j\in\NN$, let  $h_j'=u^*h_ju\,$;  then $\{h'_j\}_j$ is another family of
pairwise orthogonal  rank-one projections with sum I. We have
\[
\sum_i\left(\sum_j\gamma_j\Tr(e_{ii}\,h_j')\right)=\sum_j\gamma_j\Tr(h_j')=\sum_j\gamma_j\leq k
\]
and
\[
0\leq\sum_j\gamma_j\Tr(e_{ii}\,h_j')\leq\sum_j\Tr(e_{ii}\,h_j')=\Tr(e_{ii})=1.
\]
Since $x\geq0$ and $\gbeta\leq \gbeta'$,
\begin{equation}\label{eq:la traza de M_gx}
\Tr(M_\gbeta x)\leq\Tr(M_{\gbeta'}\,x)=\Tr(M_{\gbeta''}\,u^*\,x\,u)
=\sum_i\gbeta_i''\left(\sum_j\gamma_j\Tr(e_{ii}\,h_j')\right)
\end{equation}
Now, starting from \eqref{eq:la traza de M_gx} and applying the inequality
\eqref{eq:la desigualdad numerica}
 to the numbers $\gbeta_1''\geq\gbeta_2''\geq\cdots\geq0$ and
$\{\sum_j\gamma_j\Tr(e_{ii}\,h_j)\}_i$, we get
\begin{align*}
\Tr(M_\gbeta \,x)&\leq\sum_i\gbeta_i''\left(\sum_j\gamma_j\Tr(e_{ii}\,h_j')\right)
\leq\sum_{i=1}^k\gbeta_i''\\
&= U_k(\gbeta'')=U_k(g')<U_k(g)+\e k.
\end{align*}
As $\e$ and $x$ were arbitrary, we have proven
the reverse inequality.
\end{proof}

\begin{rem}\label{w-vn}Two operators $a,\, b \in  B(H)$ are said to be
{\it approximately unitarily equivalent} if there exists a sequence
$\{u_n\}_{n\in \NN}\subset \U(H)$
such that
\[
\lim_{n\rightarrow \infty}\|a- u_n\,b\,u_n^*\|=0\,.
\]
 This equivalence
is well-known to operator theorists and operator algebraists.
As a consequence of the Weyl-von Neumann
theorem, it follows \cite[II.4.4]{Davidson1996} that $a,\, b \in B(H)^{\rm sa}$ are
approximately unitarily equivalent if and only if their essential spectrums
(with respect to the classical Calkin algebra) coincide
and $\dim\ker(a-\lambda I)=\dim\ker(b-\lambda I)$ for every $\lambda$ that is not
in the essential spectrum of these operators. From this it can be deduced
(see the proof of \cite[II.4.4]{Davidson1996}) that for every $b\in B(H)^{+}$
and every orthonormal basis $\cB$ of $H$, there exists $M_g\in B(H)^{+}$
-- diagonal with respect to $\cB$ -- that is
approximately unitarily equivalent to $b$.
\end{rem}

The following is the main result of this section.

\begin{teo}[A contractive Schur-Horn theorem for $B(H)$]
\label{theorem:neumann contractivo}Let $H$ be a separable complex Hilbert space and
let $P_D$ denote the unique trace preserving conditional expectation onto
the discrete masa of diagonal operators with respect to the orthonormal basis $\cB$ of $H$.
Then, for $b\in B(H)^+$,
\[
\overline{\{ P_D(v\,b\,v^*):\ v\in \cC(H)\}}^{\, \|\,\|}
= \{M_f:\ f\in \ell^\infty(\NN)^+,\, M_f\prec_w b\}\,.
\]
\end{teo}
\begin{proof}
We first consider a reduction  to the case where $b$ is diagonalizable
with respect to the orthonormal basis $\cB$. Indeed, by
Remark \ref{w-vn} there exists $g\in \ell^\infty(\NN)^+$ such that $b$ and $M_g$ are
approximately unitarily equivalent. It is then straightforward to see that
\[
\overline{\{v\, b\,v^*:\ v\in \cC(H) \}}^{\, \|\,\|}
=\overline{\{v\, M_g\,v^*:\ v\in \cC(H) \}}^{\, \|\,\|}\,,
\]
and that
\begin{equation}\label{redu1}
\overline{\{ P_D(v^*\,b\,v):\ v\in \cC(H)\}}^{\, \|\,\|}
=\overline{\{ P_D(v^*\,M_g\,v):\ v\in \cC(H)\}}^{\, \|\,\|}\,.
\end{equation}
By \eqref{eq:caracterizacion del U_k de Neumann},
$U_k(b)=U_k(M_g)$ and $L_k(b)=L_k(M_g)$ for all $k\in \NN$.
These identities, together with \eqref{redu1}, imply that -- without loss of generality --
we can assume that $b=M_g$ for some $g\in\ell^\infty(\NN)^+$.

Let $v\in\cC(H)$ and let $p\in B(H)$ be a projection with $\Tr(p)=k$.
Since $vv^*\leq I$ and $0\leq P_D(p)\leq I$ we have $v^*P_D(p) \, v\in\cC(H)^+$ and
$\Tr(v^*\,P_D(p)\,v)=\Tr(P_D(p)^{1/2}\,vv^*\,P_D(p)^{1/2})\leq\Tr(P_D(p))=k$.
Put $M_\falpha=P_D(v\,M_\gbeta \,v^*)$.
Then
\begin{align*}
U_k(M_f)
&=\sup\{\Tr(P_D(v M_\gbeta v^*)\,p):\, \Tr(p)=k\}\\
&=\sup\{\Tr((v M_\gbeta v^*)\,P_D(p)):\, \Tr(p)=k\}\\
&=\sup\{\Tr( M_\gbeta\, (v^*P_D(p)\,v)):\, \Tr(p)=k\}
\leq U_k(M_\gbeta),
\end{align*}
where in the last inequality we are using Lemma \ref{lemma:U_k contractivo} and the fact
that $v^*P_D(p)\,v\in\cC(H)^+$. Thus, $M_\falpha\prec_w M_\gbeta$ and, as $U_k(\cdot)$
is norm-continuous for every $k\in \NN$, we get the inclusion ``$\subset$''.

For the reverse inclusion, assume that $M_\falpha\prec_w M_\gbeta$ (i.e., $\falpha\prec_w\gbeta$)
and let $\e>0$.
We follow the idea of the proof
of \cite[Theorem II.2.8]{Bhatia1997}.
Consider $\falpha',\gbeta'\in\ell^\infty(\NN)\oplus\ell^\infty(\NN)$, given by
\[
\falpha'=(\falpha+\e\,e)\oplus\, \e\,e,\ \ \ \gbeta'=(\gbeta+\e\,e)\oplus\, 0\,.
\]
where $e\in\ell^\infty(\NN)$ is the identity.
Note that $\|f\oplus 0 -f'\|_\infty,\,\|g\oplus 0 -g'\|_\infty< \e$. Since $f,g\geq0$, we have
$U_k(\falpha')=U_k(f)+k\e$, $U_k(\gbeta')=U_k(\gbeta)+k\e$, $L_k(\falpha')
=k\e$, $L_k(\gbeta')=0$, for all $k\in\NN$. Hence we have $\falpha'\prec \gbeta'$. By
Theorem \ref{theorem:Neumann 2.18-3.13}, there exists a unitary operator $u\in B(H\oplus H)$ such
that
\begin{equation}\label{eq:SH para los doble primas}
\|M_{\falpha'}-P_{D\oplus D}(u\, M_{\gbeta'}\,u^*)\|<\e.
\end{equation}
We have
 \begin{equation}\label{eq:aproximacion en las diagonales}
\|M_{\gbeta\oplus0}-M_{\gbeta'}\|<\e,\ \ \ \|M_{\falpha\oplus0}-M_{\falpha'}\|<\e.
\end{equation}
Now let $q=I\oplus0\in B(H\oplus H)$, and let $c=quq$ (clearly a contraction),
seen as an operator in $B(H)$. Then,
as $q\,P_{D\oplus D}=P_D\oplus0$ and
$q\,M_{\falpha\oplus0}=q\,M_{\falpha\oplus0}\,q=M_{\falpha\oplus0}$, we can use
\eqref{eq:SH para los doble primas} and \eqref{eq:aproximacion en las diagonales} to get
\begin{align*}
\|M_{\falpha}-P_D(c \,M_{\gbeta}\,c^*)\|
&=\|q(M_{\falpha\oplus0}-P_{D\oplus D}(u\,M_{\gbeta\oplus0}\,u^*))q\|\\
&\leq\|M_{\falpha\oplus0}-P_{D\oplus D}(u\,M_{\gbeta\oplus0}\,u^*)\|\\
&<2\e+\|M_{\falpha'}-P_{D\oplus D}(u \,M_{\gbeta'}\,u^*)\|<3\e.
\end{align*}
As $\e$ was arbitrary, we conclude that $M_f\in\overline{\{ P_D(v^*M_g\,v):\ v\in \cC(H)\}}^{\, \|\,\|}$.
\end{proof}

\begin{rem}
The positivity assumption  in Theorem \ref{theorem:neumann contractivo}
is not just a technicality: even in dimension one we have $-1\prec_w 0$, and
$\{v\,0\,v^*: |v|\leq1\}=\{0\}$.
\end{rem}
As a consequence of Theorem \ref{theorem:neumann contractivo}
we get that, for $f,\,g\in \ell^\infty(\NN)^+$,
\begin{equation}\label{theorem:neumann contractivoeq}
M_f\in \overline{\{P_D(v  \,M_g \, v^*): \ v\in \cC(H)\}}^{\, \|\, \|} \ \text{ if and only if } \ f\prec_w g\,.
\end{equation}

\section{Majorization in II$_\infty$-factors}\label{section:majorization}

Recall that $(\m,\tau)$ denotes a $\sigma$-finite and semi-finite diffuse
 von Neumann algebra.  Given $a\in \m^{\rm sa}$,  we consider the functions
$$U_t(a)=\int_0^t \lambda_s(a)\ ds \ \ \text{ and } \ \  L_t(a)=\int_0^t \mu_s(a)\ ds\,,
\ \ \ t\in\RR^+\, , $$
where $t\mapsto\lambda_t(a)$ and $t\mapsto\mu_t(a)$ denote  the upper and lower spectral scales
(Definition \ref{def:upper and lower spectral scale}).

Our next goal is to describe the maps $b\mapsto U_t(b)$ and $b\mapsto L_t(b)$ by means of
 \cite[Lemma 4.1]{FackKosaki1986}.
We will make use of the following relation between spectral scales and singular
 values:
 \begin{equation}\label{los mus y nus chus chus, saluds}
 \lambda_t(a)=\nu_t(a + \gamma I) - \gamma\, ,\ \  \mu_t(a)= \rho - \nu_t(-a+\rho I),  \ \ \
 a\in\m^{\rm sa},
 \end{equation}
 for any $\gamma,\,\rho\in \RR$ such that $a + \gamma I,\, -a+\rho I\in\m^+$.
We will denote by $\cP_t(\m)$ the set of all projections in $\m$ of trace $t$, i.e. \[\cP_t(\m)=\{p\in\cP(\m):\tau(p)=t\}.\]
Since $(\m,\tau)$ is diffuse and semifinite,  $\cP_t(\m)\neq \emptyset$ for every $t\geq 0$.

\begin{lem}\label{lem: u_t con proyecs}
For any $a\in \m^{\rm sa}$,
\[
U_t(a)=\sup\{\tau(a\,p):\ p\in\cP_t(\m)\}, \ \ \ L_t(a)=\inf\{\tau(a\,p):\ p\in\cP_t(\m)\},\
\ t\in\RR^+.
\]
\end{lem}
\begin{proof} The equalities are an immediate consequence of the identities
\eqref{los mus y nus chus chus, saluds} together with \cite[Lemma 4.1]{FackKosaki1986}
and the fact that, for every $t\in \RR^+$,
\[
\sup\{\tau(ap):\ p\in\cP_t(\m)\}=\sup\{\tau((a+\gamma \, I)\,p):\ p\in\cP_t(\m)\} -\gamma t\,.
\qedhere
\]
\end{proof}

\begin{rem}\label{rem:Ut caso comp} If $a\in \cK(\m)^+$,
then $\mu_t(a^+)=0$ for $t\in \RR^+$.
Let $\{e(t)\}_{t\in \RR^+}\subset \m$ be a complete flag for $a$ such that
$a=\int_0^\infty \lambda_t(a)\ de(t)$ (which exists by the assumptions on $\m$).
Then, using \cite[Proposition 2.7]{FackKosaki1986} and
\eqref{los mus y nus chus chus, saluds},  we have
\[
U_t(a)=\int_0^t \lambda_s(a)\ ds=\tau(a\, e(t))\ \ \text{ and } \ \ L_t(a)=0,
\ \ \ t\in \RR^+.
\]
Thus, for a positive $\tau$-compact operator $a$ the supremum in Lemma \ref{lem: u_t con proyecs}
is attained explicitly by means of the projection $e(t)$ in $\cP_t(\m)\cap\{a\}'$.
\end{rem}

\begin{lem}\label{lem:continuidad} Let $b\in \m^{\rm sa}$. Then, for each $t\in \RR^+$,
the functions
$b\mapsto U_t(b)$, $b\mapsto L_t(b)$ are $\|\cdot\|_1$-continuous, and they are also
$\cT$-continuous on bounded sets of $\m^{\rm sa}$ .
\end{lem}
\begin{proof} It is enough to prove the statement for $U_t(\cdot)$, since $L_t(b)=-U_t(-b)$.
Given $\e>0$, by Lemma \ref{lem: u_t con proyecs} there exists
$p\in\cP_t(\m)$ with $U_t(x)\leq\tau(xp)+\e$. Then
\[
U_t(x)-U_t(y)\leq\tau(xp)+\e-\tau(yp)\leq\|x-y\|_{(t)}+\e\leq \|x-y\|_1+\e\, ,
\] where we used the inequality $\tau((x-y)p)\leq \tau(|x-y|p)\leq\|x-y\|_{(t)}$ that follows from Lemma \ref{lem: u_t con proyecs}. By letting $\e\to0$ and reversing the roles of $x$ and $y$ we conclude the $\cT$ and $\|\cdot\|_1$ continuity of $b\mapsto U_t(b)$ on bounded sets, by Proposition \ref{proposition:la top de la medida esta dada por una norma en acotados}.
\end{proof}

From now on we will specialize $(\m,\tau)$ to be a $\sigma$-finite II$_\infty$-factor with faithful
normal semifinite tracial
weight $\tau$.

We  begin by describing the notion of majorization between selfadjoint operators in the
II$_\infty$-factor $\m$.  In the setting of non-finite von Neumann algebras,
this concept was developed for selfadjoint operators in \cite{Hiai1989}. Our presentation,
inspired by Neumann's work \cite{Neumann1999}, is fairly different (see Remark \ref{rem:tau compacto tiene L_t=0} below).

\begin{fed}\label{def:mayorizacion}
Let  $a,b\in\m^{\rm sa}$.
\begin{enumerate}
\item We say that $a$ is {\it submajorized} by $b$ (denoted $a\prec_w b$) if
\[
U_t(a)\leq U_t(b),\  \mbox{ for every }t\in\RR^+.
\]
\item We say that $a$ is {\it majorized} by $b$, denoted $a\prec b$, if $a\prec_w b$ and
\[
L_t(a)\geq L_t(b),\ \ \mbox{ for every }t\in\RR^+.
\]
\end{enumerate}
\end{fed}

\begin{rem}\label{rem:tau compacto tiene L_t=0}
If $b\in \cK(\m)^+$, then $\mu_t(b)=0$ for all $t\in \RR^+$ and
therefore $L_t(b)=0$ for all $t\in \RR^+$. Thus, if $a\in \m^+$ and $a\prec_wb$, then
$a\prec b$.

For $a,\,b\in \m^+$,
our notion of majorization is strictly stronger than the one considered in \cite{Hiai1987}.
As we have already mentioned, our notion of majorization does coincide with that of
\cite{Hiai1989} for selfadjoint operators in a II$_\infty$-factor (see Corollary \ref{corollary:Hiai}). It is worth pointing out
that in \cite{Hiai1989} majorization is described (for normal operators)
in terms of
Choquet's theory on comparison of measures, rather than in the simple terms used above:
Lemma \ref{lem: u_t con proyecs} shows that the notion of majorization in
a II$_\infty$-factor from definition \ref{def:mayorizacion}
is an analogue of the notion of operator majorization in $B(H)$
as described in Definition \ref{def:mayorizacion de Neumann}.
\end{rem}

For a fixed $b\in\m^{\rm sa}$, we write $\Omega_\m(b)$ for the set of all elements
in $\m^{\rm sa}$ that are majorized by $b$, i.e.
\[
\Omega_\m(b)=\{a\in \m^{\rm sa}:\ a\prec b\}.
\]

\begin{pro}\label{pro:mayo preservada}
Let $b\in\m^{\rm sa}$. Then  $\Omega_\m(b)$ is a
bounded $\cT$-closed convex set that contains the unitary orbit $\U_\m(b)$.
\end{pro}
\begin{proof}\noindent
For any $x\in \m^{\rm sa}$, the definition of $U_t(x)$ and $L_t(x)$, together with the
right-continuity of $\lambda_t(x)$ and $\mu_t(x)$, imply that
\[
\lim_{t\rightarrow 0^+} \frac{U_t(x)}{t}=\lambda_t(0)=\max\sigma(x)\ \
\text{ and } \ \ \lim_{t\rightarrow 0^+} \frac{L_t(x)}{t}=\mu_t(0)=\min\sigma(x).
\]
Hence, $a\prec b$ implies $\sigma(a)\subset[\min\sigma(b),\max\sigma(b)]$; in particular
$\|a\|\leq \|b\|$, so $\Omega_\m(b)$ is a bounded set.
Lemma \ref{lem:continuidad} immediately implies that it is closed in the measure topology.
Moreover, if $u\in\um$, it is easy to see that
$\lambda_t(ubu^*)=\lambda_t(b)$. So  $U_t(ubu^*)=U_t(b)$ and, similarly,
$L_t(ubu^*)=L_t(b)$. Thus $ubu^*\prec b$, and $\U_\m(b)\subset \Omega_\m(b)$.

Let $a_1,\,a_2\in \m^{\rm sa}$,  $\gamma\in [0,1]$, with $a_1\prec b$, $a_2\prec b$.
Using Lemma \ref{lem: u_t con proyecs},
\begin{align*}
U_t(\gamma\, a_1+(1-\gamma)\,a_2)
&=\sup\{\tau(p\,(\gamma \,a_1+(1-\gamma)\,a_2)):\ \tau(p)=t\}\\
&=\sup\{\gamma \,\tau(p\,a_1)+(1-\gamma)\,\tau(p\,a_2):\ \tau(p)=t\}\\
&\leq \gamma \,U_t(a_1)+(1-\gamma)\,U_t(a_2)\leq U_t(b)\,.
\end{align*}
Similarly, $$L_t(\gamma \,a_1+(1-\gamma)\,a_2)\geq \gamma\, L_t(a_1)+(1-\gamma)\, L_t(a_2)
\geq L_t(b)\, , $$ so $\gamma \,a_1+(1-\gamma)\,a_2\prec b$, and $\Omega_\m(b)$ is convex.
\end{proof}

\begin{rem}\label{definicion: bar b y underbar b}
Let $b\in\m^{\rm sa}$. The function $t\mapsto\lambda_{t}(b)$ is non-increasing and bounded;
therefore the numbers
$\lambda_{\max}^{\rm e}(b)=\lim_{t\rightarrow \infty}\lambda_t(b)$ and
$\lambda_{\min}^{\rm e}(b)=\lim_{t\rightarrow\infty}\mu_t(b)$ exist. Indeed, we have
\begin{equation}\label{eq: limites en infinito}
\lambda_{\max} ^{\rm e} (b)=\max \ \sess(b)=\lim_{t\to\infty}\frac{U_t(b)}t\,
\ , \ \ \lambda_{\min} ^{\rm e} (b)=\min \ \sess(b)=\lim_{t\to\infty}\frac{L_t(b)}t\, .
\end{equation}
Consider the operators $\bar b$, $\underline b\in\m^+$ given by
\begin{equation}\label{eq: bar b y underbar b: def}
\bar b=(b-\lambda_{\max} ^{\rm e} (b)\,I)^+ \ \text{ and } \ \underline b=(\lambda_{\min} ^{\rm e} (b)\,I-b)^+  \, .
\end{equation}
Both $\bar b,\, \underline b$ are positive $\tau$-compact operators with  orthogonal support.
It is easy to check that,
for all $t\geq0$,
$U_t(b)=U_t(\bar b)+t\ \lambda_{\max}^{\rm e}(b)$,
$L_t(b)=-U_t(\underline b)+t\ \lambda_{\min}^{\rm e}(b)$,
 and $ L_t(\underline b)=L_t(\bar b)=0$.  If $a\prec b$ then, by \eqref{eq: limites en infinito},
 $$\lambda_{\min}^{\rm e}(b)\leq \lambda_{\min}^{\rm e}(a)\leq \lambda_{\max}^{\rm e}(a)
 \leq \lambda_{\max}^{\rm e}(b).$$

\end{rem}

We finish the section with three lemmas on perturbations that will be used
in Section \ref{section:Schur-Horn Theorem}.

\begin{lem}\label{lema:previo a traza infinita en los tau-compactos}
Let $x\in\cK(\m)^{+}$,  $z\in\cP(\m)$ infinite with $zx=0$ and $\e>0$.
Then there exists $x'\in \cK(\m)^+$ such that:
\begin{enumerate}
\item the support of $x'$ contains $z$;
\item $\|x'-x\|<\e$;
\item\label{lema:previo a traza infinita en los tau-compactos:4}
$\ds\lambda_t(x')=\lambda_t(x)+\e/(6+t)$, $t\in[0,\infty)$.
\end{enumerate}
\end{lem}
\begin{proof}
Since $x$ is $\tau$-compact, there exists $s_0>0$ such that $\lambda_{s_0}(x)<\e/6$. Let
$p_1=p^{x}(\lambda_{s_0}(x),\infty)$. The $\tau$-compactness of $x$ guarantees that
$\tau(p_1)<\infty$.

As $x$ is
$\tau$-compact and positive, there exists a complete flag $e_x(t)$ with
$x=\int_0^\infty\lambda_t(x)\,de_x(t)$. Note that $p_1=e_x(s_0)$.
Let $e_1(t)$ be a complete flag over $z$, and define
\[
x'=\int_0^{s_0}\left(\lambda_t(x)+\frac\e{6+t}\right)\,de_x(t)
+\int_{0}^\infty\left(\lambda_{t+s_0}(x)+\frac\e{6+t+s_0}\right)\,de_1(t)
\]
The second term above equals $x'p_1^\perp=x'z$ and its norm is less than $\e/3$; so
\begin{align*}
\|x-x'\|&\leq\left\|\int_0^{s_0}\frac\e{6+t}\,de_x(t)\right\|+\|xp_1^\perp\|+\|x'p_1^\perp\|
<\frac\e6+\frac\e6+\frac\e3< \e
\end{align*}
It is clear by construction (since $e_x(t)e_1(s)=0$ for all $t,s$) that
\[
\lambda_t(x')=\lambda_t(x)+\frac\e{6+t},\ \ t\in[0,\infty),
\]
and this implies $x'\in\cK(\m)$.
\end{proof}

\begin{lem}\label{lema: a y b aproximados}
Let $\cA\subset\m$ be a diffuse von Neumann subalgebra.
Let $a\in\cA^{\rm sa}$, $b\in\m^{\rm sa}$ with $a\prec b$, and fix $\e>0$.
Then there exist $a'\in\cA^{\rm sa}$, $b'\in\m^{\rm sa}$ such that
\begin{enumerate}
\item $\|a-a'\|<\e$, $\|b-b'\|<\e$;
\item  $a'\prec b'$;
\item $\overline{a'}$ , $\underline{a'}\,$, $\overline{b'}$ , $\underline{b'}$ (as defined in Remark \ref{definicion: bar b y underbar b}) have infinite support.
\end{enumerate}
\end{lem}
\begin{proof}
We first consider a partition of the identity
\[
s_1=p^b[\lambda_{\max}^{\rm e}(b)+\frac\e8, \infty) \, , \ \
s_2=p^b(\lambda_{\min}^{\rm e}(b)-\frac\e8,\lambda_{\max}^{\rm e}(b)+\frac\e8)\, , \ \
s_3=p^b(-\infty,\lambda_{\min}^{\rm e}(b)-\frac\e8].
\]
The projection $s_2$ is infinite, while the others may or may not be infinite.
We consider a decomposition $s_2=z_1+z_2+z_3$ into three mutually orthogonal infinite projections,
such that
\[
z_1\leq p^b(\lambda_{\max}^{\rm e}(b)-\frac\e{8}, \lambda_{\max}^{\rm e}(b)+\frac\e{8}),\ \
z_3\leq p^b(\lambda_{\min}^{\rm e}(b)-\frac\e{8}, \lambda_{\min}^{\rm e}(b)+\frac\e{8}).
\]
Let $\underline a,\,\bar a \in \cK(\cA)^+$ and $\underline b,\,\bar b\in \cK(\m)^+$ as in
\eqref{eq: bar b y underbar b: def}.
Apply Lemma \ref{lema:previo a traza infinita en los tau-compactos} to $\bar b s_1$ with the projection
$z_1$ and to $\underline b\, s_3$ with $z_3$, to obtain $(\bar{b})'$, $(\underline {b})'\in\cK(\m)^+$,
both with infinite support and such that $\|(\bar b)'-\bar b\,s_1\|<\e/4$,
$\|(\underline b)'-\underline b\,s_3\|<\e/4$. Define
\[
b'=((\bar b)' + \lambda^{\rm e}_{\max}(b)(s_1+z_1))+(s_2-z_1-z_3)b-((\underline b)'-\lambda^{\rm e}_{\min}(b)(s_3+z_3)).
\]

As $b=(\bar b\,s_1+\lambda_{\max}^{\rm e}(b)\,s_1)+bs_2
-(\underline b\,s_3 - \lambda_{\min}^{\rm e}(b)s_3)$, we get
\begin{align*}
\|b'-b\|&\leq\|(\bar b)'-\bar b\,s_1\|+\|\lambda^{\rm e}_{\max}(b)\,z_1-b\,z_1\|
+\|\lambda^{\rm e}_{\min}(b)\,z_3-b\,z_3\|+\|(\underline b)'-\underline b\,s_3\|\\
&<\frac\e4+\frac\e4+\frac\e4+\frac\e4=\e.
\end{align*}
Note that $\lambda_{\max}^{\rm e}(b')=\lambda_{\max}^{\rm e}(b)$;
then  $\overline{b\,'}=(\bar b)'$ , $\underline{b\,'}=(\underline b)'$ have infinite support,
\begin{align}\label{eq:los lambda y mu perturbados 1}
\lambda_t(b')&=\lambda_t(\overline{b'})+\lambda_{\max}^{\rm e}(b')
=\lambda_t((\overline b)')+\lambda_{\max}^{\rm e}(b)\\
\nonumber &=\lambda_t(\overline b)+\frac\e{6+t}+\lambda_{\max}^{\rm e}(b)
=\lambda_t(b)+\frac\e{6+t}
\end{align}
and similarly $\mu_t(b')=\mu_t(b)-\frac\e{6+t}$.

Proceeding with $a$ in the same way we did for $b$, we
obtain $a'\in\cA^{\rm sa}$ with $\|a-a'\|<\e$, with $\overline{a'}$ and $\underline {a'}$ having
infinite support, and such that
\begin{equation}\label{eq:los lambda y mu perturbados 2}
\lambda_t(a')=\lambda_t(a)+\frac\e{6+t},\ \ \mu_t(a')=\mu_t(a)-\frac\e{6+t},\ \ t\in[0,\infty).
\end{equation}
From \eqref{eq:los lambda y mu perturbados 1}, \eqref{eq:los lambda y mu perturbados 2},
and the fact that $a\prec b$, we deduce that $a'\prec b'$.
\end{proof}

Let $\cN$ be a semifinite diffuse von Neumann algebra with fns trace $\tau$.
We consider the set $L^1(\cN)\cap\cN$, which consists of those $x\in\cN$ with $\|x\|_1<\infty$.
The elements in $L^1(\cN)\cap\cN$ are necessarily compact, since $\int_0^\infty\lambda_t(|x|)\,dt<\infty$ forces
$\nu_t(x)=\lambda_t(|x|)\xrightarrow[t\to\infty]{}0$.

\begin{lem}\label{lema:en los tau-compactos siempre se puede traza infinita}
Let $\cN$ be a semifinite diffuse von Neumann algebra with fns trace $\tau$, and let
$x\in L^1(\cN)^{\rm sa}$, $\e>0$. Then there exists $x'\in L^1(\cN)^{\rm sa}$ such that
\begin{enumerate}
\item\label{lema:en los tau-compactos siempre se puede traza infinita:1} $\|x'-x\|_1<\varepsilon$;
\item\label{lema:en los tau-compactos siempre se puede traza infinita:2} $\lambda_t(x')=\lambda_t(x)+\e/(10+4t^2)$;
\item\label{lema:en los tau-compactos siempre se puede traza infinita:3} $\mu_t(x')=\mu_t(x)-\e/(10+4t^2)$;
\item $\tau(p^{x'}(0,\infty))=\infty$, $\tau(p^{x'}(-\infty,0))=\infty$;
\item $p^{x'}(-\infty,0)+p^{x'}(0,\infty)=I$.
\end{enumerate}
\end{lem}
\begin{proof}
Since $x$ is $\tau$-compact, its essential spectrum contains zero. Then $\lambda_t(x)\geq0$,
$\mu_t(x)\leq0$ for all $t$. With that in mind, the proof runs as the proof of Lemma \ref{lema:previo a traza infinita en los tau-compactos},
using the $L^1$ property instead of compactness to choose $p_1$ and considering the positive
and negative parts of $x$ separately.
\end{proof}

\section{Schur-Horn theorems in II$_\infty$-factors}\label{section:Schur-Horn Theorem}

In this section we prove versions of the Schur-Horn theorem
in the $\sigma$-finite II$_\infty$-factor $(\m,\tau)$ (Theorems \ref{teo:SH} and \ref{teo:SH contractivo}),
in the spirit of Neumann's work \cite{Neumann1999}. We also consider
versions of these results for $\tau$-integrable operators (Theorems \ref{teo: sh en tipo traza} and \ref{teo: sh en tipo traza, contractivo}).

We begin with the following result, which comprises the main technical part of the proof of
Theorem \ref{teo:SH}
(by allowing us to reduce the argument to a discrete case).
Recall that $V(\ve,\delta)$ denotes the canonical basis of neighborhoods of 0 in the
measure topology, indexed by $\ve,\,\delta>0$.

\begin{pro}\label{pro:aproximacion}
Let $\cA\subset \m$ be a diffuse von Neumann subalgebra. Let
$a\in \cA^{\rm sa}$, $b\in \msa$ be such that $a\prec b$ and fix $m\in \NN$. Then
there exist $\{p_n\}_{n\geq 1}\subset \cP ( \cA),\, \{q_n\}_{n\geq 1}\subset \cP (\m)$ such that
\begin{enumerate}
\item\label{pro:aproximacion1} $p_i\,p_j=q_i\,q_j=0$ for $i\neq j$;
\item\label{pro:aproximacion2} $\tau(p_n)=\tau(q_n)=\tau(p_1)$ for all $n\in \NN$;
\item\label{pro:aproximacion3} $\tau(1-\sum_{n\geq 1}p_n)=\tau(1-\sum_{n\geq 1}q_n)< \frac{1}{m}$;
\item\label{pro:aproximacion4} there exist $\falpha,\gbeta\in\ell^\infty_\RR(\NN)$ such that:
\begin{enumerate}
\item\label{pro:aproximacion4a} $\falpha\prec\gbeta$ ;
\item\label{pro:aproximacion4b}   \[ (a-\sum_{n\geq 1} \falpha(n)\, p_n ),\  (b-\sum_{n\geq 1} \gbeta(n)\, q_n )\in V(\frac{1}{m},\frac{1}{m} ).\]
\end{enumerate}
\end{enumerate}
\end{pro}

\begin{proof}
By Lemma \ref{lema: a y b aproximados} there exist $a'\in\cA^{\rm sa}$, $b'\in\m^{\rm sa}$ with $\|a-a'\|<1/2m$,
$\|b-b'\|<1/2m$, $a'\prec b'$, and such that $\bar a$, $\underline a$, $\bar b$, $\underline b$ (as defined in Remark \ref{definicion: bar b y underbar b}) have infinite support.
 So, at the cost of replacing $1/m$ with $2/m$ in \eqref{pro:aproximacion4b} above, we can assume without loss
of generality that
 $\tau(r_1)=\tau(s_1)=\tau(r_3)=\tau(s_3)=\infty$, where $r_1,s_1,r_3,s_3\in \cP(\cM)$ are as in the proof of Lemma \ref{lema: a y b aproximados}.

Since $\cA$ is diffuse, there exist complete flags
$\{e_{\bar a}(t)\}_{t\in[0,\infty)}$, $\{e_{\underline a}(t)\}_{t\in[0,\infty)}$ in $\cA$
over $r_1$ and $r_3$ respectively
such that $\tau(e_{\bar a}(t))=\tau(e_{\underline a}(t))=t$  for $t\geq 0$ and
\[
\bar a=\int_0 ^\infty \lambda_s({\bar a})\ de_{\bar a}(s)\, , \
\underline a= \int_0 ^\infty \lambda_s({\underline a})\ de_{\underline a}(s).
\]
Similarly, there exist complete flags $\{e_{\bar b}(t)\}_{t\in[0,\infty)}$,
$\{e_{\underline b}(t)\}_{t\in[0,\infty)}$ over $s_1$ and $s_3$ respectively
such that $\tau(e_{\bar b}(t))=\tau(e_{\underline b}(t))=t$  for $t\geq 0$ and
\[
\bar b=\int_0 ^\infty \lambda_s({\bar b})\ de_{\bar b}(s)\, , \
\underline b= \int_0 ^\infty \lambda_s({\underline b})\ de_{\underline b}(s).
\]
Let $q_t=I-(e_{\bar b}(t)+e_{\underline b}(t))$,
$p_t=I-(e_{\bar a}(t)+e_{\underline a}(t))$. Then $\{q_t\},\{p_t\}$ are decreasing nets of
projections that converge strongly to $s_2$, $r_2$ respectively.
For the rest of the proof,
we will fix $t> 0$
big enough so that the following three properties hold (all guaranteed by the fact that
$\lambda_t(x)\to0$ as $t\to\infty$ if $x\in\cK(\m)$):
\begin{equation}\label{eq:condiciones para el t_1}
\left(\lambda_{\min}^{\rm e}(b) - \frac{1}{m}\right)q_t\,\leq\,b\,q_t \, \leq \left(\lambda_{\max}^{\rm e}(b)+ \frac{1}{m}\right)q_t\, .
\end{equation}
\begin{equation}\label{eq:condiciones para el t_2}
\left(\lambda_{\min}^{\rm e}(b) - \frac{1}{m}\right)p_t\,\leq\,a\,p_t \, \leq \left(\lambda_{\max}^{\rm e}(b)+ \frac{1}{m}\right)p_t\, .
\end{equation}
\begin{equation}\label{eq:condiciones para el t_3}
\max\{\lambda_t({\bar a}), \lambda_t({\bar b}),
\lambda_t({\underline a}), \lambda_t({\underline b})\}\,<\,\frac1m.
\end{equation}

Now apply \cite[Lemma 3.2.]{ArgeramiMassey2007} and Corollary
\ref{lem:topologia de la medida en un II_1} to $a\,e_{\bar
a}(t)$ in the II$_1$ factor $e_{\bar a}(t)\m e_{\bar a}(t)$ and
to $a\,e_{\underline a}(t)$ in the II$_1$-factor
$e_{\underline a}(t)\m e_{\underline a}(t)$. This way
we get $N\in\NN$ with $N\geq t\cdot 3\,m\cdot
(2\,\|b\|\,m+3)$,  partitions $\{p_j\}_{j=1}^N$ and
$\{p'_j\}_{j=1}^N$ of $e_{\bar a}(t)$ and $e_{\underline a}(t)$ respectively given by
\[
p_j=e_{\bar a}\left(\frac{j \, t}{N}\right )-e_{\bar a}\left(\frac{(j-1) \, t}{N}\right)\, , \
p'_j=e_{\underline a}\left(\frac{j \, t}{N}\right )-e_{\underline a}\left(\frac{(j-1) \, t}{N}\right)
\, , \ 1\leq j\leq N,
\]
and coefficients $\alpha'_1\geq\alpha'_2\geq\cdots\geq\alpha'_N$,
 $\alpha''_1\geq\alpha''_2\geq\cdots\geq\alpha''_N$
 given by
\[
\alpha_j'=\frac{N}{t}\int_{(j-1)t/N}^{jt/N}\lambda_s({ae_{\bar a}(t))}\,ds
=\frac{N}{t}\tau(ap_j),\ \ \alpha_j''=\frac{N}{t}\tau(ap_j'),
\]
such that
\begin{equation}\label{eq:discretos en un entorno1}
(a \, e_{\bar a} (t)- \sum_{j=1}^N \alpha'_j\, p_j)\, , \
(a \, e_{\underline a} (t)- \sum_{j=1}^N \alpha''_j\, p'_j)\, \in V(\frac{1}{m},\frac{1}{2\,m})\,
\end{equation}
(recall that $\|x\|_{(1)}\leq\|x\|_1$ and that
if $\|x\|_{(1)}<1/4m^2$, then $x\in V(1/2m,1/2m)$; see the proof
of Proposition \ref{proposition:la top de la medida esta dada por una norma en acotados}).
Similarly, we obtain for $b$
 partitions $\{q_j\}_{j=1}^N$ and $\{q'_j\}_{j=1}^N$ of $e_{\bar b}(t)$ and
 $e_{\underline b}(t)$ respectively such that
 \[
 q_j=e_{\bar b}\left(\frac{j \, t}{N}\right)-e_{\bar b}\left(\frac{(j-1) \, t}{N}\right)\, , \
 q'_j=e_{\underline b}\left(\frac{j \, t}{N}\right)-e_{\underline b}\left(\frac{(j-1) \, t}{N}\right)
 \, , \ 1\leq j\leq N,
 \]
 and coefficients $\beta'_1\geq\beta'_2\geq\cdots\geq\beta'_N$,
 $\beta''_1\geq\beta''_2\geq\cdots\geq\beta''_N$ given by
\[
\beta_j'=\frac{N}{t}\tau(bq_j),
\ \ \beta_j''=\frac{N}{t}\tau(bq_j')
\]
 with
\begin{equation}\label{eq:discretos en un entorno2}
(b \, e_{\bar b} (t)- \sum_{j=1}^N \beta'_j\, q_j)\, , \
(b \, e_{\underline b} (t)- \sum_{j=1}^N \beta''_j\, q'_j)\, \in V(\frac{1}{m},\frac{1}{2\,m})\, .
\end{equation}
Consider now a partition $\{I_j\}_{j=1}^{L}$ of
$[\lambda_{\min}^{\rm e}(b)-\frac{1}{m},\lambda_{\max}^{\rm e}(b)+\frac{1}{m}]$ into $L$ consecutive disjoint sub-intervals
with $2\leq L\leq 2\,\|b\|\,m+3$,
with $I_1=[\lambda_{\min}^{\rm e}(b)-\frac{1}{m},\lambda_{\min}^{\rm e}(b))$,
$I_L=(\lambda_{\max}^{\rm e}(b),\lambda_{\max}^{\rm e}(b)+\frac{1}{m}]$, and
such that the length of each $I_j$ is no greater than $\frac{1}{m}$.
Define
\[
a_e=p_t\,a,\ \ \
b_e=q_t\, b.
\]
Let $\gamma_1=\lambda_{\min}^{\rm e}(b)$, $\gamma_L=\lambda_{\max}^{\rm e}(b)$, and choose
$\gamma_j\in I_j$ for
$2\leq j\leq L-1$. The choice of the $\gamma_j$, together with  \eqref{eq:condiciones para el t_1} and
\eqref{eq:condiciones para el t_2}, imply that
\begin{equation}\label{eq:aproximacion de discretos para el esencial}
\|\,a_e- \sum_{j=1}^L \gamma_j \, p^{a_e}(I_j)\,\|< \frac{1}{m}\, , \ \
\|\,b_e- \sum_{j=1}^L \gamma_j \, p^{b_e}(I_j)\,\|< \frac{1}{m}\, .
\end{equation}
For $j\in\{1,\ldots,L\}$ let
\[
t^a_j=\left\{\begin{array}{cl}
\lfloor \,\frac{\tau(p^{a_e}(I_j))\, N}{t} \,\rfloor  &\text{ if } \tau(p^{a_e}(I_j))<\infty\\ \ \\
\infty &\text{ if } \tau(p^{a_e}(I_j))=\infty\, ,
\end{array}\right.
\]
where $\lfloor x\rfloor $ denotes the integer part of $x\in \RR$.
We construct $\{t^b_j\}_{j=1}^L$ in the same way.
For each $j$, if $t^a_j=\infty$ we consider a partition
$\{p^{(j)}_i\}_{i\in\NN}\subset \cP(\cA)$ of $p^{a_e}(I_j)$ with $\tau(p^{(j)}_i)=\frac{t}{N}$ for
all $i\in \NN$;
otherwise, if $t^a_j<\infty$,
we consider a partition
$\{p^{(j)}_i\}_{i=1}^{t^a_j+1}\subset \cP(\cA)$ with $\tau(p^{(j)}_i)=\frac{t}{N}$
for $1\leq i\leq t^a_j$, and $\tau(p^{(j)}_{t^a_j+1})<\frac{t}{N}$.

Analogously, we consider partitions $\{q^{(j)}_i\}_{i}\subset \cP(\m)$ of $p^{b_e}(I_j)$ for $1\leq j\leq L$.
Since $\overline b$ and $\underline b$ have infinite support,
 \begin{equation}\label{eq:cotas entre los coeficientes}
 t^b_1=t^b_L=\infty\, , \ \ \lambda_{\min}^{\rm e}(b) \leq\min_{1\leq j\leq L}\gamma_j\leq \max_{1\leq j\leq L}\gamma_j\leq \lambda_{\max}^{\rm e}(b)
 \end{equation}  and  there exists $i_0\in\{1,\ldots,L\}$ with $t^a_{i_0}=\infty$.
 As $L\leq 2\|b\|m+3$,
$N\geq t\cdot 3\,m\cdot (2\,\|b\|\,m+3)$,
\begin{equation}\label{eq:colas}
\sum_{j:\ t^a_j<\infty} \tau(p^{(j)}_{t^a_j +1})\leq\sum_{i=1}^L \frac{t}{N}\leq \frac{1}{3\,m}\, , \  \sum_{j:\ t^b_j<\infty} \tau(q^{(j)}_{t^b_j +1})\leq \frac{1}{3\,m}.
\end{equation}
We can assume that the two projections
$\sum_{j:\ t^a_j<\infty} p^{(j)}_{t^a_j +1}$,
$\sum_{j:\ t^b_j<\infty} q^{(j)}_{t^b_j +1}$
have equal trace; indeed we can
take the necessary mass (which will be certainly less than $1/2m$)
from one of the projections
$p^{a_e}(I_{i_0})$, $p^{b_e}(I_{L})$ respectively (since each of them is an infinite projection)
before considering the partitions  of these projections (this, at the cost of replacing the
``$\|\cdot\|<1/m$'' in \eqref{eq:aproximacion de discretos para el esencial} by ``$\in V(1/m,1/2m)$'').
From \eqref{eq:aproximacion de discretos para el esencial} and
\eqref{eq:colas},
\begin{equation}\label{en un entorno3}
(a_e-\sum_{j=1}^L\gamma_j\,\sum_{i=1}^{t^a_j} p^{(j)}_i)\, ,\ \
(b_e-\sum_{j=1}^L\gamma_j\,\sum_{i=1}^{t^b_j} q^{(j)}_i) \, \in V(\frac{1}{m},\frac{1}{m})\,.
\end{equation}
Let $\{(\alpha_i,p_i)\}_{i\geq 1}$ be an enumeration of the countable set
\begin{align*}
\{(\alpha'_j,p_j):\, 1\leq j\leq N \}&\cup\, \{(\alpha''_j,p'_j):\, 1\leq j\leq N \}\\
&\cup\,\{(\gamma_j,p^{(j)}_i) :\, 1\leq j\leq L\, ,\, 1\leq i\leq t^a_j\}
\end{align*}
and let
$\{(\beta_i,q_i)\}_{i\geq 1}$ be an enumeration of the countable set
\begin{align*}
\{(\beta'_j,q_j):\, 1\leq j\leq N \}&\cup\, \{(\beta''_j,q'_j):\, 1\leq j\leq N \}\\
&\cup\,\{(\gamma_j,q^{(j)}_i) :\, 1\leq j\leq L\, ,\, 1\leq i\leq t^b_j\}.
\end{align*} By construction, $\{p_n\}_{n\in \NN}\subset \cA$.
It also follows that \eqref{pro:aproximacion1}, \eqref{pro:aproximacion2}, and
\eqref{pro:aproximacion3} in the statement of the Theorem hold. Moreover,
from \eqref{eq:discretos en un entorno1}, \eqref{eq:discretos en un entorno2} and
\eqref{en un entorno3} we get part b) of \eqref{pro:aproximacion4} (with
$\falpha=\{\alpha_n\}_{n\geq 1}$, ${\gbeta}=\{\beta_n\}_{n\geq 1}$).
It remains to show that $\falpha\prec {\gbeta}$
in the sense of Definition \ref{def:mayorizacion de Neumann}. We
will only prove that $U_k(\falpha)\leq U_k(\gbeta)$ for $k\geq 1$, since the $L_k$
inequalities follow in a similar way. We have
 \begin{eqnarray*}
U_k({\gbeta })&=&\left\{\begin{array}{ll}
\sum_{i=1}^k\beta'_j
&\text{ if } 1\leq k\leq N\, ,\\
\sum_{i=1}^N\beta'_j+(k-N)\lambda_{\max}^{\rm e}(b)
 &\text{ if } N<k\,
\end{array}\right.
\end{eqnarray*}
(recall that $\gamma_L=\lambda_{\max}^{\rm e}(b)$ and that there is
an infinity of $\gamma_L$ in the list
$\{\beta_n\}$)).
For $U_k(\falpha)$ we get
 \begin{eqnarray*}
U_k({\falpha })&=&\left\{\begin{array}{ll}
\sum_{i=1}^{k}\alpha'_j
&\text{ if } 1\leq k\leq N\, ,\\
\sum_{i=1}^{N}\alpha'_j+\sum_{i=N+1}^k \gamma_{\sigma(i)}
 &\text{ if } N<k\,
\end{array}\right.
\end{eqnarray*}
for appropriate choices $\sigma(i)\in\{1,\ldots,L\}$.
If $1\leq k\leq N$, then
\begin{align*}
U_k(g)&=\sum_{i=1}^k \beta'_i=\frac N {t} \ \int_0^{\frac{kt}{N}}\lambda_s(b)\ ds
=\frac{N}{t}\,U_{kt/N}(b)\\
&\geq\frac{N}{t}\,U_{kt/N}(a)= \frac N {t} \ \int_0^{\frac{kt}{N}}\lambda_s(a)\ ds=\sum_{i=1}^k \alpha'_i=U_k(f).
\end{align*}
If $N<k$,
\begin{align*}
U_k(g)&=\frac N {t} \ \int_0^{t}\lambda_s(b)\ ds + (k-N)\lambda_{\max}^{\rm e}(b)\\
& \geq \frac N {t} \ \int_0^{t}\lambda_s(a)\ ds+\sum_{i=N+1}^k \gamma_{\sigma(i)}=U_k(f)
 \end{align*}
 since, by \eqref{eq:cotas entre los coeficientes},
 $\gamma_{\sigma(i)}\leq \lambda_{\max}^{\rm e}(b)$ for all $i$.
\end{proof}

\begin{rem}\label{rem submayo}

Let $\cA\subset \m$ be a diffuse von Neumann subalgebra. Fix
$a\in \cA^{+}$, $b\in \m^+$ such that $a\prec_w b$ and let $m\in \NN$.
Then a slightly modified version of the proof of Proposition \ref{pro:aproximacion}
(with $r_3=s_3=0$, $\lambda_{\min}^{\rm e}(b)=\lambda_{\min}^{\rm e}(a)=0$) shows that
there exist $\{p_n\}_{n\geq 1}\subset \cP ( \cA),\, \{q_n\}_{n\geq 1}\subset \cP ( \m)$
and $\falpha,\gbeta\in\ell^\infty(\NN)^+$ such that conditions
\eqref{pro:aproximacion1}-\eqref{pro:aproximacion3} and \eqref{pro:aproximacion4b}
hold, and such that $f\prec_w g$.
We will use these facts for the proof of the contractive Schur-Horn theorem
(Theorem \ref{teo:SH contractivo}).
\end{rem}

\begin{lem}\label{lem:embedding}
Let $\cN\subset\m$ be a von Neumann subalgebra, with $E_\cN$ the unique trace-preserving conditional
expectation onto $\cN$. Let $\{p_j\}_{j\in \NN}\subset\cZ(\cN)$ be a family of
mutually orthogonal projections,
pairwise equivalent in $\m$. Let  $\{e_{ij}\}$  be a system of matrix units in $B(H)$.
Then there exists a (possibly non-unital) normal *-monomorphism
$\pi:B(H)\to\m$ such that
\begin{equation}\label{eq:embedding1}
\pi(e_{jj})=p_j,\ \ j\in \NN,
\end{equation}
\begin{equation}\label{eq:embedding2}
E_\cN(\pi(x))=\pi(P_D(x)),\ \ x\in B(H).
\end{equation}
\end{lem}
\begin{proof}
Let $p=\sum_jp_j$. Since the projections $\{p_j\}_j$ are mutually
orthogonal and equivalent, they can be extended to a system of
matrix units $\{p_{ij}\}_{i,j\in \NN}$ in $p\m p$, with
$p_{ij}^*=p_{ji}$, $p_{ij} p_{kh}=\delta_{jk}\,p_{ih}$,
$p_{jj}=p_j$ for all $i,j,k,h\in \NN$. Also, it is easy to check
(using that $p_j\in\cZ(\cN)$ for all $j$)
that $E_\cN(p_{ij})=\delta_{ij}\,p_j$.

Since $p$ is an infinite projection, $p\m p$ is a II$_\infty$-factor.
It is standard that $p\m p\simeq
B(H)\otimes \,p_{11}\m p_{11}$ via the (normal) *-isomorphism
\[
\eta:y\mapsto\,\sum_{i,j}\,e_{ij}\otimes\,p_{1i}\ y\  p_{j1}.
\]
In particular, $\eta(p_{ij})=e_{ij}\otimes p_{11}$.
Now let $\pi:B(H)\to \m$ be given by $\pi(x)=\eta^{-1}(x\otimes p_{11})$. So $\pi(e_{ij})=p_{ij}$ for all
$i,j$.
For any $x\in B(H)$, we have
$x=\sum_{i,j}x_{ij}\,e_{ij}$ for coefficients $x_{ij}\in\CC$. If $\pi(x)=0$, then for all $i,j$ we have
\begin{align*}
0&=p_{ij}\, \pi(x)\,p_{ij}=\pi(e_{ij})\,\pi(x)\,\pi(e_{ij})=\pi(e_{ij}\,x\,e_{ij})=\pi(x_{ij}\,e_{ij})=x_{ij}\,p_{ij}.
\end{align*}
So $x_{ij}=0$ for all $i,j$ and this shows that $\pi$ is a monomorphism. Finally,
using the normality of $\pi$ and
$E_\cN$,
\begin{align*}
\pi(P_D(x))&=\pi\left(\sum_j x_{jj}\ e_{jj}\right)=\sum_jx_{jj}\ \pi(e_{jj})
=\sum_jx_{jj}\ p_{jj}\\
&=\sum_{i,j}x_{ij}\ E_\cN(p_{ij})=E_\cN\left(\sum_{i,j}\ x_{ij}\ p_{ij}\right)
=E_\cN(\pi(x)).\qedhere
\end{align*}
\end{proof}

The characterization of $U_t$ in Lemma \ref{lem: u_t con proyecs} allows us
to prove that conditional expectations are ``contractive'' from  a majorization point of view:

\begin{lem}\label{proposition:la esperanza es mayorizada}
Let $\cA\subset \m$ be a diffuse abelian von Neumann subalgebra.
Then, for every $b\in\m^{\rm sa}$, we have $E_\cA(b)\prec b$.
\end{lem}
\begin{proof}
Fix $t>0$ and let $\e>0$.
Then we can apply Lemma \ref{lem: u_t con proyecs} in $\cA$ to get a projection $q\in\cP(\cA)$
with $\tau(q)=t$ and such that $U_t(E_\cA(b))\leq\tau(E_\cA(b)\,q)+\e$.
Since $\tau(E_\cA(b)\,q)=\tau(E_\cA(b\,q))=\tau(b\,q)\leq U_t(b)$, we conclude that
 $U_t(E_\cA(b))\leq U_t(b)+\e$ for all
$\e>0$; so, $U_t(E_\cA(b))\leq U_t(b)$. Applying the same proof to $-b$,
we get $L_t(E_\cA(b))=-U_t(E_\cA(-b))\geq-U_t(-b))=L_t(b)$. As $t$ was arbitrary,
we get  $E_\cA(b)\prec b$.
\end{proof}

We are finally in position to state and prove our main theorem.

\begin{teo}[A Schur-Horn theorem for II$_\infty$-factors]\label{teo:SH}
Let $\cA\subset\m$ be a diffuse abelian von Neumann subalgebra. Then, for any
$b\in\m^{\rm sa}$,
\[
\overline{E_\cA(\um(b))}^\cT=\{a\in\cA^{\rm sa}:\ a\prec b\}.
\]
\end{teo}
\begin{proof}
By Proposition \ref{pro:mayo preservada} and Lemma \ref{proposition:la esperanza es mayorizada},
$\overline{E_\cA(\um(b))}^\cT\subset\{a\in\cA:\ a\prec b\}$.
To show the reverse inclusion, fix $a\in\cA^{\rm sa}$ with $a\prec b$ and fix $m\in\NN$.
Applying Proposition \ref{pro:aproximacion}  to $a,b$ we obtain
sequences $\falpha=\{\alpha_n\}$, $\gbeta=\{\beta_n\}\subset\ell^\infty_\RR(\NN)$,
$\{p_n\}\subset\cP(\cA)$,
$\{q_n\}\subset\cP(\m)$ with
\begin{equation}\label{eq:SH:1}
p_i\ p_j=q_i\ q_j=0\mbox{ if }i\ne j; \ \ \tau(p_1)=\tau(p_j)=\tau(q_j)\mbox{ for all }j;
\end{equation}
\begin{equation}\label{eq:SH:2}
\tau(1-\sum_{n\geq 1}p_n)=\tau(1-\sum_{n\geq 1}q_n)< \frac{1}{m};
\end{equation}
\begin{equation}\label{eq:aproximaciones}
(a-\sum_{n\geq 1} \alpha_n\, p_n ),
\  (b-\sum_{n\geq 1} \beta_n\, q_n )\in V(\frac{1}{m},\frac{1}{m} );
\end{equation}
\begin{equation*}
\falpha\prec\gbeta.
\end{equation*}
By Theorem \ref{theorem:Neumann 2.18-3.13}
there exists a unitary $v\in B(H)$
such that
\begin{equation*}\label{eq:matrices}
\|M_\falpha-P_D(v \, M_\gbeta \,v^*)\|<\frac1{m}.
\end{equation*}
The conditions on the projections in \eqref{eq:SH:1} and \eqref{eq:SH:2}
guarantee that we can
choose $w\in\um$ with $w\,q_n\,w^*=p_n$ for all $n$. Let $p=\sum_np_n$, $q=\sum_nq_n$; then by
\eqref{eq:SH:2} there
exists a partial isometry $z\in\m$ with $z^*z=p^\perp$, $zz^*=q^\perp$. Let $u$ be the unitary
$u=(\pi(v)+z)\,w$, where $\pi$ is the *-monomorphism
from Lemma \ref{lem:embedding} with respect to the projections $\{p_n\}_n$.
From \eqref{eq:aproximaciones},
\begin{equation*}\label{eq:SH_1}
a-\pi(M_\falpha)\in V(\frac1m,\frac1{m}), \ \ \ w\,b\,w^*-\pi(M_\gbeta)\in V(\frac1m,\frac1{m}).
\end{equation*}
Note that  by \eqref{eq:SH:2} we have $\tau(p^\perp)<1/m$, $\tau(q^\perp)<1/m$, so
$z,z^*\in V(\e,1/m)$ for any $\e>0$. From this we conclude that
\[
(\pi(v)+z)\ \pi(M_g)\ (\pi(v)+z)^*-\pi(vM_g v^*)\in V(\e,\frac2m),\ \ \e>0.
\]It follows that
\begin{equation*}\label{eq:SH 1.5}
u\,b\,u^*-\pi(v\,M_g\,v^*)\in V(\frac2m,\frac{3}m).
\end{equation*}
Letting $m$ vary all along $\NN$, we have constructed sequences of unitaries $\{u_m\}_m\subset\m$ and
$\{v_m\}_m\subset \U(H)$, and sequences $\{f_m\}_m,\{g_m\}_m\subset\ell^\infty_\RR(\NN)$ with
\begin{equation}\label{eq:SH 1.6}
\pi(M_{\falpha_m})-a\xrightarrow[m\to\infty]\cT0, \ \
M_{\falpha_m}-P_D(v_{m}\, M_{\gbeta_m}\, v_{m}^*)\xrightarrow[m\to\infty]{\|\ \|}0,
\end{equation}
\[
u_m\,b\,u_m^*-\pi(v_{m}\,M_{\gbeta_m}\,v_{m}^*)\xrightarrow[m\to\infty]\cT0.
\]
Using that $\pi$ is a *-monomorphism, the $\cT$-continuity of $E_\cA$ (Corollary \ref{lem:propiedades
de T}) and the fact that $E_\cA\circ\pi=\pi\circ P_D$ (Lemma \ref{lem:embedding}) we get from \eqref{eq:SH 1.6} that
\begin{equation}\label{eq:SH_2}
\pi(M_{\falpha_m})-\pi(P_D(v_{m}\,M_{\gbeta_m} \,v_{m}^*))\xrightarrow[m\to\infty]{\| \ \|}0,
\end{equation}
\begin{equation}\label{eq:SH_3}
E_\cA(u_m\,b\,u_m^*)-\pi(P_D(v_{m}\,M_{\gbeta_m} \,v_{m}^*))\xrightarrow[m\to\infty]\cT0.
\end{equation}
From \eqref{eq:SH 1.6}, \eqref{eq:SH_2}, and \eqref{eq:SH_3},
\[
E(u_m\,b\,u_m^*)-a\xrightarrow[m\to\infty]\cT0.
\]
That is, $a\in\overline{E_\cA(\um(b))}^\cT$.
\end{proof}

\begin{rem}It is natural to ask whether one can remove the closure bar in the description of the set
$\{a\in\cA^{\rm sa}: a\prec b\}$ given in Theorem \ref{teo:SH}. Next we show an example in
which
\begin{equation*}\label{ole}
E_\cA(\um(b))\subset E_\cA(\overline{\um(b)}^\cT)\subsetneq\overline{E_\cA(\um(b))}^\cT \,.
\end{equation*}
This implies that the characterization of $\{a\in\cA^{\rm sa}: a\prec b\}$ given in
Theorem \ref{teo:SH} cannot be strengthened in the II$_\infty$ case.

We consider $p\in\cP(\m)$ an infinite projection with $p^\perp$ also infinite.
Then $U_t(p)=t$, $L_t(p)=0$ for all $t$. Since $U_t(I)=t$, $L_t(I)=t$, we have $I\prec p\,$; then
 \begin{equation}\label{un ejem}
 I\in \overline{E_\cA(\um(p))}^{\cT} \ \text{ but } \ I\not\in E_\cA(\overline{\um(p)}^\cT)\, .\end{equation}
 Indeed, Theorem \ref{teo:SH} guarantees the claim to the left in \eqref{un ejem}.
 On the other hand, assume that there exists $x\in\overline{ \U_\m(p)}^\mathcal T$ with $I=E_\cA(x)$.
 By Corollary \ref{lem:propiedades de T}, $0\leq x\leq I$ and then
\[
0=\tau(I-E_\cA(x))=\tau(E_\cA(I-x))=\tau(I-x)\,.
\]
This last fact implies that $I=x\in \overline{ \U_\m(p)}^\mathcal T$ by the faithfulness of $\tau$. But as $\|\cdot\|_{(1)}$ is a unitarily invariant norm,
for any $u\in\um$ we get
\[
\|I-u\,p\,u^*\|_{(1)}=\|u\,(I-p)\,u^*\|_{(1)}=\|I-p\|_{(1)}>0
\]
as $p\ne I$. Since $\|\cdot\|_{(1)}$ is $\cT$-continuous
(see Proposition \ref{proposition:la top de la medida esta dada por una norma en acotados}),
there is positive distance from $I$ to the $\cT$-closure of the
unitary orbit of $p$, a contradiction.

It would be interesting to have a description of the set
$E_\cA(\overline{\um(b)}^\cT)$ for an abelian diffuse von Neumann subalgebra of a general
$\sigma$-finite semifinite factor $(\m,\tau)$. But even in the I$_\infty$ factor case this
problem is known to be hard (see \cite[Thm 15]{Kadison2003}, \cite{arv2006,ArvesonKadison2007}
for further discussion). In the II$_1$-factor case Arveson and Kadison
\cite{ArvesonKadison2007} conjectured that
\begin{equation}\label{eq conj}
E_\cA\left( \overline{ \U_\m(b)}^\mathcal T \right)= \{a\in\cA^{\rm sa}: a\prec b\}\,,
\end{equation}
which is still an open problem (see \cite{ArgeramiMassey2007,ArgeramiMassey2008a,ArgeramiMassey2009} for a detailed discussion).
\qed\end{rem}

The following result shows that the notion of majorization in $\m^{\rm sa}$ from
 Definition \ref{def:mayorizacion}
coincides with the majorization introduced by Hiai in \cite{Hiai1989}.
Thus, several other characterizations of majorization
can be obtained from  Hiai's work. Following Hiai, we say that a map
is \emph{doubly stochastic} if it is unital, positive and preserves the trace.

\begin{cor}\label{corollary:Hiai}
Let $\cA\subset\m$ be a diffuse abelian von Neumann subalgebra and let $a,b\in\m^{\rm sa}$. Then the following statements are equivalent:
\begin{enumerate}
\item\label{corollary:Hiai:1} $a\prec b$;
\item\label{corollary:Hiai:2} $a\in\overline{E_\cA(\um(b))}^\cT$;
\item\label{corollary:Hiai:3} $a\in\overline{\co\{\um(b)\}}^\cT$;
\item\label{corollary:Hiai:4} there exists a doubly stochastic map $F$ on $\m$ with $a=F(b)$;
\item\label{corollary:Hiai:5} there exists a completely positive doubly stochastic map $F$ on $\m$ with $a=F(b)$;
\item\label{corollary:Hiai:6} $\tau(f(a))\leq\tau(f(b))$ for every convex function $f:I\rightarrow [0,\infty)$ with $\sigma(a),\,\sigma(b)\subset I$.
\item\label{corollary:Hiai:7} $a$ is spectrally majorized by $b$ in the sense of \cite{Hiai1989}.
\end{enumerate}
\end{cor}
\begin{proof}
By Theorem \ref{teo:SH}, \eqref{corollary:Hiai:1} and \eqref{corollary:Hiai:2} are equivalent.
The statements \eqref{corollary:Hiai:3}-\eqref{corollary:Hiai:7}
are mutually equivalent by \cite[Theorem 2.2]{Hiai1989}. Also,
\eqref{corollary:Hiai:3} implies \eqref{corollary:Hiai:1}
by Proposition \ref{pro:mayo preservada}. So it will be enough to
show that \eqref{corollary:Hiai:1} implies \eqref{corollary:Hiai:4}.

Let $a\in\cA$ with $a\prec b$. By Theorem
\ref{teo:SH}, there exist unitaries $\{u_j\}\subset\m$ such that $a=\lim_\cT E_\cA(u_jbu_j^*)$.
Consider the sequence of completely positive contractions
$E_\cA(u_j\cdot u_j^*):\m\to\cA$; by compactness in the BW topology \cite[Theorem 7.4]{Paulsen2002},
this sequence
admits a convergent (pointwise ultraweakly)
subnet $\{E_\cA(u_{j_k}\cdot u_{j_k}^*)\}$.
Let $F$ be the limit of such subnet. Since $a=\lim_\cT E_\cA(u_jbu_j^*)$ and
$F(b)=\lim_{\sigma-\mbox{wot}}E_\cA(u_{j_k}b u_{j_k}^*)$, we conclude (mimicking the argument
in the proof of Lemma 3.3 in \cite{Hiai1989}) that $F(b)=a$. It is easy to check that
$F$ is unital and that it preserves the trace.
\end{proof}

We finish this section with contractive and $L^1$ analogs of Theorem \ref{teo:SH}.

\begin{teo}\label{teo:SH contractivo}
Let $\cA\subset\m$ be a diffuse abelian von Neumann subalgebra and let $b\in\m^{+}$. Then
\begin{equation}\label{teo:SH contractivo:1}
\overline{E_\cA(\{c\,b\,c^*: \|c\|\leq1\})}^\cT=\{a\in\cA^+:\ a\prec_w b\}.
\end{equation}
\end{teo}
\begin{proof}
If $c\in\m$ is a contraction, then $\lambda_t({c\,b\,c^*})\leq\lambda_t(b)$
\cite[Lemma 2.5]{FackKosaki1986}. So $c\,b\,c^*\prec_w b$
and then Lemmas \ref{proposition:la esperanza es mayorizada}
and \ref{lem:continuidad}
give the inclusion
``$\subset$'' above.

For the reverse inclusion, the proof runs exactly as that of Theorem \ref{teo:SH}, but instead of using Proposition \ref{pro:aproximacion}
and \eqref{theorem:Neumann 2.18-3.13d} to obtain a sequence of
unitary operators in $\m$, we use \eqref{theorem:neumann contractivoeq} and Remark \ref{rem submayo} to obtain a convenient sequence of contractions in $\m$.
\end{proof}

\begin{rem}
The positivity condition in Theorem \ref{teo:SH contractivo} cannot be relaxed to
selfadjointness. As a trivial example, take $b=0$; then $-I\prec_w b$, but $c\,b\,c^*=0$ for all $c$,
so the set on the left in \eqref{teo:SH contractivo:1} is $\{0\}$.
\end{rem}

Recall that $L^1(\m)\cap\m$ consists of those $x\in\m$ with $\tau(|x|)<\infty$,
and that such elements are necessarily $\tau$-compact.

\begin{teo}\label{teo: sh en tipo traza} Let $\cA\subset\m$ be a diffuse abelian von Neumann subalgebra and
let $b\in L^1(\m)\cap\m^{\rm sa}$. Then
\[
\overline{E_\cA(\um(b))}^{\, \|\cdot\|_1}=\{a\in L^1(\m)\cap\cA^{\rm sa}:\ a\prec b, \ \tau(a)=\tau(b)\}
\]
\end{teo}
\begin{proof}
 Proposition \ref{pro:mayo preservada} together with Lemma
 \ref{proposition:la esperanza es mayorizada} show that
${E_\cA(\um(b))}\subset\{a\in\cA^{\rm sa}:\ a\prec b,\ \tau(a)=\tau(b)\}$.
Then Lemma \ref{lem:continuidad} and the $\|\cdot\|_1$-continuity of the trace
imply the inclusion of the corresponding closure.

Conversely, suppose that $a\prec b$ and $\tau(a)=\tau(b)$. First assume that $b\in\m^+$.
Then $a\in\cA^+$.
By Theorem
\ref{teo:SH}, there exists a sequence of unitaries $\{u_j\}$ such that
$E_\cA(u_j\,b\,u_j^*)\xrightarrow{\cT}a$. Since $b$ is positive,
$\|E_\cA(u_j\,b\,u_j^*)\|_1=\tau(E_\cA(u_j\,b\,u_j^*))=\tau(b)=\tau(a)=\|a\|_1$. Then
\cite[Theorem 3.7]{FackKosaki1986} guarantees that
$\|E_\cA(u_j\,b\,u_j^*)-a\|_1\to0$.

If $b$ is not positive, we apply Lemma
\ref{lema:en los tau-compactos siempre se puede traza infinita}
to obtain $a'\in\cA$, $b'\in\m$, with
\begin{enumerate}
\item $a'\prec b'$,
\item $\|a'-a\|_1<\varepsilon$, $\|b'-b\|_1<\varepsilon$;
\item $\tau(p^{a'}(0,\infty))=\tau(p^{b'}(0,\infty))=\infty$;
\item $\tau(p^{a'}(-\infty,0))=\tau(p^{b'}(-\infty,0))=\infty$;
\item $p^{a'}(-\infty,0)+p^{a'}(0,\infty)=p^{b'}(-\infty,0)+p^{b'}(0,\infty)=I$.
\end{enumerate}
Let $r_1=p^{a'_+}(0,\infty)$, $r_2=p^{a'_-}(0,\infty)$.
The last three conditions above guarantee that we can find a unitary $v\in\um$ with
\[
v\,(p^{b'_+}(0,\infty))\,v^*=r_1,\
v\,(p^{b'_-}(0,\infty))\,v^*=r_2.
\] Let $b''=vb'v^*$. Then $a'\prec b''$.
Since both are $\tau$-compact, we deduce that $a'_+\prec b''_+$, $a'_-\prec b''_-$.
Note that $a'_+,b''_+\in r_1\m r_1$, $a'_-,b''_-\in r_2\m r_2$. As both $r_1,r_2\in\cA$
are infinite projections, the factors $r_1\m r_1$ and $r_2\m r_2$ are II$_\infty$.
So we can apply the first part of the proof to obtain unitaries
$\{u^{(1)}_j\}\subset\cU(r_1\m r_1)$, $\{u^{(2)}_j\}\subset\cU(r_2\m r_2)$, with
\[
\|E_\cA(u^{(1)}_j\ b''_+\ (u^{(1)}_j)^*)-a'_+\|_1\to0,\ \ \
\|E_\cA(u^{(2)}_j\ b''_-\ (u^{(2)}_j)^*)-a'_-\|_1\to0
\]
Since $r_1+r_2=I$, $r_1r_2=0$, the operators $u_j=(u^{(1)}_j+u^{(2)}_j)v$ are unitaries
in $\m$. Then
\begin{align*}
\|E_\cA(u_j\,b\,u_j^*)-a\|_1&\leq\|E_\cA(u_j\,b\,u_j^*)-E_\cA(u_j\,b'\,u_j^*)\|_1
+\|E_\cA(u_j\,b'\,u_j^*)-a'\|_1+\|a'-a\|_1\\
&\leq\|b'-b\|_1+\|a'-a\|_1+\|E_\cA(u^{(1)}_j\,b''\,(u^{(1)}_j)^*)-a'_+\|_1\\
&\ \ \ +\|E_\cA(u^{(2)}_j\,b''\,(u^{(2)}_j)^*)-a'_-\|_1\\
&\leq2\e+\|E_\cA(u^{(1)}_j\,b''_+\,(u^{(1)}_j)^*)-a'_+\|_1
+\|E_\cA(u^{(2)}_j\,b''_-\,(u^{(2)}_j)^*)-a'_-\|_1.
\end{align*}
So $\limsup_j\|E_\cA(u_j\,b\,u_j^*)-a\|_1<2\e$, and as $\e$ was arbitrary we conclude that
$\lim_j\|E_\cA(u_j\,b\,u_j^*)-a\|_1=0$, i.e. $a\in\overline{E_\cA(\um(b))}^{\,\|\cdot\|_1}$.
\end{proof}

\begin{rem}
The condition $\tau(a)=\tau(b)$ in Theorem \ref{teo: sh en tipo traza}
cannot be removed because of the $\|\cdot\|_1$-continuity of the trace $\tau$.
Actually, below we characterize the case where the trace restriction is removed but only in the case of positive operators.
\end{rem}

\begin{teo}\label{teo: sh en tipo traza, contractivo}
Let $\cA\subset\m$ be a diffuse abelian von Neumann subalgebra and
let $b\in L^1(\m)\cap\m^+$. Then
\[
\overline{E_\cA(\{c\,b\,c^*: \|c\|\leq1\})}^{\,\|\cdot\|_1}
=\{a\in\cA^+:\ a\prec_w b\}
=\{a\in\cA^+:\ a\prec b\}.
\]
\end{teo}
\begin{proof}
If $b\in L^1(\m)\cap\m^+$ and $a\prec_w b$ then,
since $\lambda_t(b)\in L^1(\RR^+)$, we get $\lambda_t(a)\in L^1(\RR^+)$. In particular,
$a\in\cK(\m)^+$. Thus, the second equality is immediate from the fact that for positive $\tau$-compact operators one has
$L_t=0$. So for the rest of the proof we focus on the first equality.

The inclusion ``$\subset$'' is obtained by combining the arguments at the beginning of
 the proofs of Theorems \ref{teo:SH contractivo} and \ref{teo: sh en tipo traza}.

Conversely, let $a\prec_wb$ for some $a\in \cA^+$ (so that $a\in\cK(\cA)^+$).
 We write both $a$ and $b$ in terms of complete flags in
 $\cA$ and $\m$ respectively, i.e.
 \[
a=\int_0^\infty\lambda_t(a)\,de_a(t),\ \ \ b=\int_0^\infty\lambda_t(b)\,de_b(t),
 \]
 with $e_a(t)\in\cA$ for all $t$ (this can be done since $\cA$ is diffuse).
 Then $a\prec_wb$ means that, for any $s>0$, $\int_0^s\lambda_t(a)\,dt\leq\int_0^s\lambda_t(b)\,dt$.
 For each $s>0$, let $p_s=e_a(s)\vee e_b(s)$, a finite projection. So we have
 $ae_a(s)\prec_wbe_b(s)$ in the II$_1$-factor $p_s\m p_s$.
 By \cite[Theorem 3.4]{ArgeramiMassey2008a}, there exists a contraction $c_s\in p_s\m p_s\subset \m$
 with
 \[
k_s:=\tau_s(|a\,e_a(s)-E_{\cA e_a(s)}\,(c_s \,e_b(s)\, b\, e_b(s)\, c_s^*)|)<\frac1{\tau(p_s)^2}.
 \]
 The trace $\tau_s$ is given by $\tau_s=\tau/\tau(p_s)$; using the fact that $e_a(s)\in\cA$
 and that $\cA$ is abelian, we get that $E_{\cA \, e_a(s)}(\cdot)=e_a(s)\,E_\cA(\cdot) $. So
 \[
\tau(|a\,e_a(s)-E_\cA(e_a(s)\,c_s \,e_b(s)\, b \,e_b(s)\, c_s^*\, e_a(s))|)=\tau(p_s)\,k_s<\frac1{\tau(p_s)}
\leq\frac1s
 \]
 (note that $p_s\geq e_a(s)$, so $\tau(p_s)\geq s$).
Let $\e>0$; fix
$s>0$ such that $s>2/\e$ and $\int_s^\infty\lambda_t(a)\,dt<\e/2$. Put $c=e_a(s)\, c_s \, e_b(s)$, a contraction
in $\m$. Then
\begin{align*}
\|a-E_\cA(c\,b\,c^*)\|_1&\leq\|a-a\,e_a(s)\|_1+\|a\,e_a(s)-E_\cA(e_a(s)\,c_s \,e_b(s)\, b\, e_b(s) \,c_s^*\,e_a(s))\|_1\\
&=\int_s^\infty\lambda_a(t)\,dt+\tau(|a\,e_a(s)-E_\cA(e_a(s)\,c_s \,e_b(s) \,b \,e_b(s)\, c_s^*\,e_a(s))|)\\
&\leq\frac\e2+\frac1s<\frac\e2+\frac\e2=\e.
\end{align*}
As $\e$ was arbitrary, this shows that $a\in\overline{E_\cA(\{c\,b\,c^*: \|c\|\leq1\})}^{\,\|\cdot\|_1}$.
\end{proof}

\begin{rem} The proof of Theorem \ref{teo: sh en tipo traza, contractivo} uses
a reduction to a II$_1$ case, under the hypothesis that the operators belong to $L^1(\m)$.
This last assumption seems to be essential for such a reduction, and there is no immediate
hope of using the same idea to obtain results like Theorems \ref{teo:SH} and
\ref{teo:SH contractivo}. Conversely, one cannot expect to use those results to obtain
Theorem \ref{teo: sh en tipo traza, contractivo}, since convergence in measure does
not imply $\|\cdot\|_1$-convergence.
\end{rem}

\bibliographystyle{abbrv}
\bibliography{sh_bib}

\begin{thebibliography}{10}

\bibitem{AntezanaMasseyRuizStojanoff2009}
J.~Antezana, P.~Massey, M.~Ruiz, and D.~Stojanoff.
\newblock The {S}chur-{H}orn theorem for operators and frames with prescribed
  norms and frame operator.
\newblock {\em Illinois J. Math.}, 51(2):537--560 (electronic), 2007.

\bibitem{ArgeramiMassey2007}
M.~Argerami and P.~Massey.
\newblock A {S}chur-{H}orn theorem in {${\rm II}\sb 1$} factors.
\newblock {\em Indiana Univ. Math. J.}, 56(5):2051--2059, 2007.

\bibitem{ArgeramiMassey2008a}
M.~Argerami and P.~Massey.
\newblock A contractive version of a {S}chur-{H}orn theorem in {$\rm II\sb 1$}
  factors.
\newblock {\em J. Math. Anal. Appl.}, 337(1):231--238, 2008.

\bibitem{ArgeramiMassey2008b}
M.~Argerami and P.~Massey.
\newblock The local form of doubly stochastic maps and joint majorization in
  {${\rm II}\sb 1$} factors.
\newblock {\em Integral Equations Operator Theory}, 61(1):1--19, 2008.

\bibitem{ArgeramiMassey2009}
M.~Argerami and P.~Massey.
\newblock Towards the carpenter's theorem.
\newblock {\em Proc. Amer. Math. Soc.}, 137(11):3679--3687, 2009.

\bibitem{arv2006}
W.~Arveson.
\newblock Diagonals of normal operators with finite spectrum.
\newblock {\em Proc. Natl. Acad. Sci. USA}, 104(4):1152--1158 (electronic),
  2007.

\bibitem{ArvesonKadison2007}
W.~Arveson and R.~V. Kadison.
\newblock Diagonals of self-adjoint operators.
\newblock In {\em Operator theory, operator algebras, and applications}, volume
  414 of {\em Contemp. Math.}, pages 247--263. Amer. Math. Soc., Providence,
  RI, 2006.

\bibitem{Bhatia1997}
R.~Bhatia.
\newblock {\em Matrix analysis}, volume 169 of {\em Graduate Texts in
  Mathematics}.
\newblock Springer-Verlag, New York, 1997.

\bibitem{Birkhoff1946}
G.~Birkhoff.
\newblock Three observations on linear algebra.
\newblock {\em Univ. Nac. Tucum\'an. Revista A.}, 5:147--151, 1946.

\bibitem{Davidson1996}
K.~R. Davidson.
\newblock {\em {$C\sp *$}-algebras by example}, volume~6 of {\em Fields
  Institute Monographs}.
\newblock American Mathematical Society, Providence, RI, 1996.

\bibitem{MR2176806}
I.~S. Dhillon, R.~W. Heath, Jr., M.~A. Sustik, and J.~A. Tropp.
\newblock Generalized finite algorithms for constructing {H}ermitian matrices
  with prescribed diagonal and spectrum.
\newblock {\em SIAM J. Matrix Anal. Appl.}, 27(1):61--71 (electronic), 2005.

\bibitem{Fack1982}
T.~Fack.
\newblock Sur la notion de valeur caract\'eristique.
\newblock {\em J. Operator Theory}, 7(2):307--333, 1982.

\bibitem{FackKosaki1986}
T.~Fack and H.~Kosaki.
\newblock Generalized {$s$}-numbers of {$\tau$}-measurable operators.
\newblock {\em Pacific J. Math.}, 123(2):269--300, 1986.

\bibitem{HardyLittlewoodPolya1929}
G.~H. Hardy, J.~E. Littlewood, and G.~P{\'o}lya.
\newblock Some simple inequalities satisfied by convex functions.
\newblock {\em Messenger of Math.}

\bibitem{Hiai1987}
F.~Hiai.
\newblock Majorization and stochastic maps in von {N}eumann algebras.
\newblock {\em J. Math. Anal. Appl.}, 127(1):18--48, 1987.

\bibitem{Hiai1989}
F.~Hiai.
\newblock Spectral majorization between normal operators in von {N}eumann
  algebras.
\newblock In {\em Operator algebras and operator theory ({C}raiova, 1989)},
  volume 271 of {\em Pitman Res. Notes Math. Ser.}, pages 78--115. Longman Sci.
  Tech., Harlow, 1992.

\bibitem{HiaiNakamura1987}
F.~Hiai and Y.~Nakamura.
\newblock Majorizations for generalized {$s$}-numbers in semifinite von
  {N}eumann algebras.
\newblock {\em Math. Z.}, 195(1):17--27, 1987.

\bibitem{Horn1954}
A.~Horn.
\newblock Doubly stochastic matrices and the diagonal of a rotation matrix.
\newblock {\em Amer. J. Math.}, 76:620--630, 1954.

\bibitem{Kadison2003}
R.~V. Kadison.
\newblock The {P}ythagorean theorem. {II}. {T}he infinite discrete case.
\newblock {\em Proc. Natl. Acad. Sci. USA}, 99(8):5217--5222 (electronic),
  2002.

\bibitem{Kadison2004}
R.~V. Kadison.
\newblock Non-commutative conditional expectations and their applications.
\newblock In {\em Operator algebras, quantization, and noncommutative
  geometry}, volume 365 of {\em Contemp. Math.}, pages 143--179. Amer. Math.
  Soc., Providence, RI, 2004.

\bibitem{MR2436756}
V.~Kaftal and G.~Weiss.
\newblock A survey on the interplay between arithmetic mean ideals, traces,
  lattices of operator ideals, and an infinite {S}chur-{H}orn majorization
  theorem.
\newblock In {\em Hot topics in operator theory}, volume~9 of {\em Theta Ser.
  Adv. Math.}, pages 101--135. Theta, Bucharest, 2008.

\bibitem{Kaf2}
V.~Kaftal and G.~Weiss.
\newblock An infinite dimensional {S}chur-{H}orn theorem and majorization
  theory.
\newblock {\em J. Funct. Anal.}, 259(12):3115--3162, 2010.

\bibitem{Kamei1983}
E.~Kamei.
\newblock Majorization in finite factors.
\newblock {\em Math. Japon.}, 28(4):495--499, 1983.

\bibitem{Kamei1984}
E.~Kamei.
\newblock Double stochasticity in finite factors.
\newblock {\em Math. Japon.}, 29(6):903--907, 1984.

\bibitem{MR2581231}
P.~Massey and M.~Ruiz.
\newblock Minimization of convex functionals over frame operators.
\newblock {\em Adv. Comput. Math.}, 32(2):131--153, 2010.

\bibitem{Neumann1999}
A.~Neumann.
\newblock An infinite-dimensional version of the {S}chur-{H}orn convexity
  theorem.
\newblock {\em J. Funct. Anal.}, 161(2):418--451, 1999.

\bibitem{MR1887630}
A.~Neumann.
\newblock An infinite dimensional version of the {K}ostant convexity theorem.
\newblock {\em J. Funct. Anal.}, 189(1):80--131, 2002.

\bibitem{Paulsen2002}
V.~Paulsen.
\newblock {\em Completely bounded maps and operator algebras}, volume~78 of
  {\em Cambridge Studies in Advanced Mathematics}.
\newblock Cambridge University Press, Cambridge, 2002.

\bibitem{Petz1985}
D.~Petz.
\newblock Spectral scale of selfadjoint operators and trace inequalities.
\newblock {\em J. Math. Anal. Appl.}, 109(1):74--82, 1985.

\bibitem{Schur1923}
I.~Schur.
\newblock {\"U}ber eine klasse von mittelbildungen mit anwendung auf die
  determinantentheorie.
\newblock {\em S.-Ber. Berliner math. Ges.}, 2:9--20, 1923.

\end{thebibliography}

\end{document}